\documentclass[11pt,reqno]{amsart}
\usepackage{graphicx}%
\usepackage{multirow}%
\usepackage{amsmath,amssymb,amsfonts}%
\usepackage{amsthm}%
\usepackage{mathrsfs}%
\usepackage{subfigure}
\usepackage[title]{appendix}%
\usepackage{xcolor}%
\usepackage{textcomp}%
\usepackage{manyfoot}%
\usepackage{comment}
\usepackage{enumitem}
\usepackage{booktabs}%
\usepackage{algorithm}%
\usepackage[colorinlistoftodos,prependcaption,textsize=tiny]{todonotes}
\usepackage{algorithmicx}%
\usepackage{algpseudocode}%
\usepackage{listings}%
\usepackage{upgreek}

\usepackage{epsfig}
\usepackage{geometry}
\usepackage{setspace}
\usepackage{tikz,tikz-cd}
\usepackage{pgfplots}
\pgfplotsset{compat=newest}
\usepgfplotslibrary{fillbetween}
\usetikzlibrary{positioning}
\usetikzlibrary{decorations.pathreplacing}
\usetikzlibrary{decorations,arrows}
\usetikzlibrary{decorations.markings}
\usetikzlibrary{patterns}
\usetikzlibrary{quotes,angles}
\usetikzlibrary{arrows}
\usepackage{bm,bbm}
\usepackage{subcaption}
\graphicspath{ {./fig/} }
\usepackage{cleveref}
\def\eu{\ensuremath{\mathrm{e}}}
\def\iu{\ensuremath{\mathrm{i}}}
\def\du{\ensuremath{\mathrm{d}}}

\newtheorem{theorem}{Theorem}[section]

\newtheorem{assump}{Assumption}[section]
\newtheorem{proposition}{Proposition}[section]%
\newtheorem{lemma}{Lemma}[section]

\newtheorem{example}{Example}[section]%
\newtheorem{remark}{Remark}[section]%

\newtheorem{definition}{Definition}[section]%

\numberwithin{equation}{section}

\begin{document}
	\title[Resolvent Convergence and Patch Approximation for Guided Modes]{Resolvent Convergence and Patch Approximation for Subwavelength Guided Modes in Non-Periodic Systems of High-Contrast Resonators}
	\author{Habib Ammari}
	\address{ETH Z\"urich, Department of Mathematics, Rämistrasse 101, 8092 Z\"urich, Switzerland}
	\email{habib.ammari@math.ethz.ch} 
	\author{Borui Miao}
	\address{Yau Mathematical Sciences Center, Tsinghua University, 100084 Beijing, China}
	\email{mbr@mail.tsinghua.edu.cn}  
	\author{Jiayu Qiu}
	\address{ETH Z\"urich, Department of Mathematics, Rämistrasse 101, 8092 Z\"urich, Switzerland}
	\email{jiayu.qiu@sam.math.ethz.ch}
	
	\date{}
	
	

	\begin{abstract}
		This paper develops, analyzes, and validates a fast algorithm for computing guided modes within bent interfaces and non-periodic defects in high-contrast resonator crystals, where the Floquet--Bloch theory is not applicable. We first establish the resolvent convergence of the governing continuous operator to the discrete capacitance operator. This result rigorously justifies the reduction of the continuous spectral problem to a discrete eigenvalue problem. Then, we develop a truncation scheme of the discrete operator, named the patch approximation, and derive a rigorous error estimate for the patch approximation. Finally, we validate the accuracy and efficiency of our scheme through various examples. Our framework provides a general, computationally efficient, and rigorously justified approach to simulate guided modes in non-periodic systems of high-contrast resonators.
	\end{abstract}

	\maketitle
	
	\section{Introduction}
	
	Subwavelength periodic structures are arrangements of high-contrast resonators, which, due to their low-frequency resonances, can couple to waves with wavelengths much larger than the typical size of a single resonator, therefore allowing the manipulation of waves at subwavelength scales. In particular, efficient waveguiding at subwavelength scales can be achieved when defects are introduced into these periodic structures. Various applications have been reported in the field of waveguiding at subwavelength scales such as signal processing, telecommunications, and analog computing, as well as emerging technologies in quantum science, biomedicine, and super-resolution sensing \cite{refguiding1,refguiding2,bookguiding}. 
	
	In this paper, we consider crystals of high-contrast resonators and introduce a fast computational method for computing defect modes, guided modes in line defects, and guided modes along bent interfaces at subwavelength scales. For line defects, the Floquet--Bloch transform is applied to reduce the problem to a strip, as in \cite{jems2021,line-defect-borui,santosa}, which greatly saves computational resources. However, for bent interfaces or more general non-periodic defective structures, such a reduction is no longer available, which leads to considerable theoretical and computational difficulties. On the one hand, in such cases, there is no rigorous theory justifying the reduction of the continuous resonance problem to its discrete counterpart. On the other hand, one cannot rely on quasi-periodic problems on strips to compute the spectrum and associated modes. Therefore, we need to develop a different framework that does not require periodicity of the resonators. 
	
	Our approach consists of two main steps. First, we derive a discrete approximation of the continuous subwavelength spectral problem in terms of the capacitance operator. This approach generalizes to non-periodic systems the asymptotic approach 
	developed in \cite{ammari2024functional,ammari2025mathematical} for periodic structures where the spectral properties of the continuous operator describing these structures are approximated at the subwavelength regime by the spectral properties of a discrete operator as the contrast between the material parameters inside and outside the resonators goes to zero.   
	Second, we approximate this discrete operator by a local truncation procedure, called the patch approximation, in order to reduce the computational cost while retaining controllable accuracy. Both the discrete approximation of the continuous problem and the patch approximation of the capacitance operator are rigorously justified. 
	Specifically, for the first step, we prove the convergence of the resolvent of the continuous operator to the capacitance operator as the contrast parameter goes to zero. For the second step, we estimate the truncation error by deriving an operator-norm error bound for the patch approximation. In contrast to the characterizations of the defect and guided modes provided in \cite{anderson, jems2021}, the spectral analysis of the capacitance operator gives a more direct approach that does not rely on the periodicity of the overall structure. The resulting numerical method, supported by these theoretical results, is efficient and significantly reduces computational costs compared to methods such as the supercell method or the method based on the Dirichlet-to-Neumann map \cite{sofiane,supercell2,supercell1,sonia}. Indeed, directly calculating the band structure of the crystal involves high computational costs owing to the fine mesh required by the high-contrast regime and the size of the supercell. We illustrate the performance of our method through a variety of examples, including bent waveguides and interfaces between two generalized honeycomb-structured materials. 
	
	Let us now highlight the main contributions of this paper, which are as follows.
	\begin{enumerate}
		\item[(i)] \emph{Resolvent convergence of the subwavelength resonance problem}: In \Cref{thm_cont_to_disc_resolvent_converge}, we establish the norm-resolvent convergence of the operator that governs the subwavelength resonance problem to the discrete capacitance operator $\mathcal{C}$. After a rescaling, the resolvent convergence holds in operator norm, as the contrast parameter $\delta$ tends to zero. This result provides a rigorous justification of the discrete approximation to the high-contrast continuous problem. \Cref{exmp_resolvent_converge_periodic} and \Cref{exmp_resolvent_converge_interface} are provided to illustrate the application to approximate the defect modes and the interface modes.  
		\item[(ii)] \emph{Patch approximation of the capacitance operator}: To efficiently assemble the discrete capacitance operator, we formulate the patch approximation, which replaces the full boundary value problem by finitely many local problems on patches. \Cref{thm:op_error_capacitance} shows that the difference between the full capacitance operator $\mathcal{C}$ and its patch approximation $\widetilde{\mathcal{C}}$ satisfies
		\[
		\|\mathcal{C}-\widetilde{\mathcal{C}}\|_{\mathcal{B}(\ell^2(\Sigma))}
		\leq C \widetilde{\rho}^{M},
		\quad 0<\widetilde{\rho}<1,\]
		where $M$ is the size of the patch. Various numerical examples are given in \Cref{sec:num_patch} to illustrate the numerical analysis and the acceleration of the patch approximation.
	\end{enumerate}
	
	Our paper is organized as follows. In \Cref{sec:general}, we introduce the subwavelength spectral problem and state the resolvent convergence theorem. In \Cref{sec_heuristics}, we describe the main ideas of the proof and discuss its implications for defect and interface modes. In particular, we show how it rigorously justifies the discrete approximations of defect modes and guided modes in both straight and curved waveguides. In \Cref{sec:PatchApprox}, we introduce the patch approximation of the capacitance operator, with the corresponding error estimates proved in \Cref{sec:err-approx}. In \Cref{sec:num_patch}, we present a variety of numerical examples to illustrate the efficiency of our method. The appendices contain technical results on the exterior Dirichlet-to-Neumann map, the proof of the resolvent convergence theorem, decay estimates for the capacitance operator, and implementation details for assembling the capacitance operator.
	
	It is worth emphasizing that our patch approximation scheme provides a low-cost and rigorously justified approach for computing guided modes in non-periodic high-contrast resonator structures. Moreover, the scheme is efficient in computing guided modes in meta-structures, since it retains the dominant interactions between resonators near the defect or interface when performing the continuous-to-discrete reduction. Therefore, the numerical method developed here can be used for the optimal design of waveguides in meta-structures \cite{design1,design2}. It can also be adapted to compute skin effect modes in non-Hermitian models \cite{Ammari2024,skin3d} and space-time localized or guided modes in time-modulated crystals \cite{erik,liora}. Furthermore, as shown in \cite{fabry3}, the capacitance operator formulation extends beyond the subwavelength regime to compute the Fabry-P\'erot band functions. Generalizing the patch approximation for computing guided modes beyond the subwavelength regime would be interesting in many applications. This would be the subject of future work.

	\section{Subwavelength spectral problem and convergence of the resolvent} \label{sec:general}
	\subsection{General setting}
	We consider a two-dimensional resonator system consisting with $N_{\mathrm{type}}$ types of resonators. Specifically, let $D^{(1)},D^{(2)},\ldots, D^{(N_{\mathrm{type}})}$ be simply connected domains with $C^1$-boundaries. For simplicity, we suppose that the area of each reference resonator is normalized, \emph{i.e.}, $|D^{(n)}|=1$ for all $n=1,\ldots,N_{\mathrm{type}}$. Since there are only finitely many resonator types, there exist constants $0<\ell_{-}\le  \ell_{+}<\infty$ such that $\ell_{-}\le  |\partial D^{(n)}|\le  \ell_{+}$ for all $1\le  n\le N_{\mathrm{type}}$. Associated with each $n\in\{1,2,\ldots,N_{\mathrm{type}}\}$, there is a discrete set $\Lambda^{(n)}\subset \mathbb{R}^2$ that denotes the translations of the corresponding reference resonator $D^{(n)}$:
	\begin{equation*}
		\mathcal{D}^{(n)}:= \bigcup_{\bm{u}\in \Lambda^{(n)}}(\bm{u} +D^{(n)}).
	\end{equation*}
	We let $\mathcal{D}$ be the union of the resonators:
	\[\mathcal{D}:= \bigcup_{1\le  n\le  N_{\mathrm{type}}}\mathcal{D}^{(n)}.\]
	We further assume that the translations $\mathcal{D}^{(n)}$ are non-overlapping, and there are sufficiently many resonators in the plane $\mathbb{R}^2$ in the following sense.
	\begin{assump} \label{asmp_geometry_assumption}
		The translations of resonators $\{ D^{(n)} \}_{n=1}^{N_{\mathrm{type}}}$ are non-overlapping, that is, 
		\begin{gather}\label{eq_geometry_assumption_4}
			\begin{aligned}
				\inf_{n\neq n'}\operatorname{dist}\big(\mathcal{D}^{(n)},\mathcal{D}^{(n')}\big)&\ge c>0,\\
				\min_{1\le n\le N_{\mathrm{type}}} \Big\{\inf_{\bm{u}\neq \bm{u}'\in \Lambda^{(n)}} \operatorname{dist}(\bm{u} +D^{(n)},\bm{u}' +D^{(n)} )\Big\}& \ge c>0.
			\end{aligned}
		\end{gather} 
		Moreover, there exist $M_{\mathrm{poly}}$ polygons $Y^{(1)},Y^{(2)},\cdots,Y^{(M_{\mathrm{poly}})}$ and associated discrete sets \[\Sigma^{(1)},\Sigma^{(2)},\cdots,\Sigma^{(M_{\mathrm{poly}})}\subset \Lambda,\] 
		such that they form a partition of $\mathbb{R}^2$:
		\begin{align} 
			(Y^{(m)}+\Sigma^{(m)})&\cap (Y^{(m')}+\Sigma^{(m')})=\emptyset,\quad 1\le m\neq m' \le M_{\mathrm{poly}},\label{eq_geometry_assumption_1}\\
			&\bigcup_{1\le  m\le  M_{\mathrm{poly}}}\overline{(Y^{(m)}+\Sigma^{(m)})}=\mathbb{R}^2 ,\label{eq_geometry_assumption_2}
		\end{align}
		where the set $Y^{(m)}+\Sigma^{(m)}$ denotes the union
		\[
		Y^{(m)}+\Sigma^{(m)}:=\bigcup_{\bm v\in\Sigma^{(m)}}(\bm v+Y^{(m)}).
		\]
		And the primitive lattice $\Lambda$ is defined by 
		\[ \Lambda:= \{m\mathring{\bm{v}}_{_1}+n\mathring{\bm{v}}_{_2}:(m,n)\in\mathbb{Z}^2\} . \]
		Furthermore, each translated polygon $Y^{(m)} + \Sigma^{(m)}$ contains exactly $K$ resonators that are well-separated from the boundary of the polygon. 
		That is, for all $\bm{v}\in \Sigma^{(m)},$ and $1\le  m\le  M_{\mathrm{poly}}$,
		\begin{equation} \label{eq_geometry_assumption_3}
			(\bm{v}+Y^{(m)})\cap \mathcal{D}\Subset \bm{v}+Y^{(m)}, \quad\text{and}\quad \operatorname{dist}\big( \partial( \bm{v}+Y^{(m)}), \mathcal{D} \big) \ge c>0 .
		\end{equation}
	\end{assump}
	From the above assumptions, we label each resonator in the polygons $Y^{(m)}_{\bm{v}}:=\bm{v}+Y^{(m)}$ by a global index $\sigma=(m,\bm{v},k)$, with $1\leq m\leq M_{\mathrm{poly}}$, $\bm{v}\in\Sigma^{(m)}$ and $1\le k\le K$. We denote the corresponding resonator by $D_{\sigma}:=D_{m,\bm{v},k}$. Thus, the collection of resonators $\mathcal{D}$ can be equivalently written as
	\[ \mathcal{D} = \bigcup_{\sigma\in \Sigma} D_{\sigma}, \]
	where the index set $\Sigma$ is given by
	\begin{equation}
		\Sigma:=\bigcup_{1\le  m\le  M_{\mathrm{poly}}}\big\{\sigma=(m,\bm{v},k):\, \bm{v}\in \Sigma^{(m)},1\le k\le K\big\}.
	\end{equation}
	
	Assumption \eqref{eq_geometry_assumption_3} ensures that there exists a resonator in arbitrary directions, no matter how far away from the origin. Physically, such an `appearance of resonators at infinity' leads to multiple scattering of waves when we impose the Dirichlet boundary condition on the resonator boundaries. This mechanism guarantees the exponential off-diagonal decay of the capacitance operator; see \Cref{rmk_exp_decay} and \Cref{apsec:expdecay}. By contrast, when this condition fails, the capacitance operator is expected to exhibit only power-law decay; see, \emph{e.g.}, \cite{Ammari2023,Ammari2024b}. 
	
	\begin{example}[Defect modes]\label{exmp:DefectModes}
		Take 
		\[ D^{(1)} = B(0,0.3),\quad D^{(2)} = B(0,0.15), \]
		and consider the square lattice $\Lambda $, where 
		\[ 
		\mathring{\bm{v}}_{_1} = (1,0)^{\top},\quad \mathring{\bm{v}}_{_2} = (0,1)^{\top}.
		\]
		We set the discrete sets as  
		\[
		\Lambda^{(1)} = \Lambda\setminus\{\mathbf{0}\} ,\quad \Lambda^{(2)} = \{\mathbf{0}\}.
		\]
		
		There exists a polygon $Y^{(1)}=\{s\mathring{\bm{v}}_{_1}+t\mathring{\bm{v}}_{_2}:s,t\in[-1/2,1/2)\}$ and a set $\Sigma^{(1)} = \Lambda$ such that the translated polygons form a partition of $\mathbb{R}^2$. It is not difficult to verify that the non-overlapping assumption \eqref{eq_geometry_assumption_1}, and the separation assumption \eqref{eq_geometry_assumption_3} hold. Moreover, such systems exhibit defect modes, and we discuss them in a more general setting in \Cref{exmp_resolvent_converge_periodic}.
	\end{example}
	
	\begin{example}[Interface modes]\label{exmp_interfaceMode}
		Another important case concerns topological interface modes. Basically, it is known that adjoining two materials with distinct topological properties can produce robust localized states near the interface, as a manifestation of the bulk-edge correspondence principle \cite{qizhang11topo_insulator,bernivig13topo_insulator}. As an example, in this paper we consider generalized honeycomb structures, as discussed in detail in \cite{Wu2015,Miao2024,Cao2025,Ammari2026}. To start, we take two unit vectors that generate the lattice
		\[
		\mathring{\bm{v}}_{_1}=(\sqrt{3}/2,-1/2)^\top,\quad
		\mathring{\bm{v}}_{_2}=(\sqrt{3}/2,1/2)^\top,
		\]
		We let 
		\[
		D^{(i)}=B(0,r)+(1+\varepsilon)\bm{d}_i,\quad
		D^{(i+6)}=B(0,r)+(1-\varepsilon)\bm{d}_i,\quad 1\le i\le 6,
		\]
		where $r<1/6$ and $0<\varepsilon\ll 1$ are chosen so that each resonator $D^{(n)}$ lies within the polygon 
		\[ Y^{(1)}=\{s\mathring{\bm{v}}_{_1}+t\mathring{\bm{v}}_{_2}:s,t\in[-1/2,1/2)\} .\] 
		Here, the vectors $\{ \bm{d}_{i} \}_{i=1}^6$ are given as follows
		\begin{gather*}
			\begin{aligned}
				\bm{d}_1:=\frac{1}{3}\mathring{\bm{v}}_{_2},\quad \bm{d}_2:=-\frac{1}{3}\mathring{\bm{v}}_{_1}+\frac{1}{3}\mathring{\bm{v}}_{_2},\quad \bm{d}_3:=\frac{1}{3}\mathring{\bm{v}}_{_1},\\
				\bm{d}_4:=-\frac{1}{3}\mathring{\bm{v}}_{_1},\quad \bm{d}_5:=\frac{1}{3}\mathring{\bm{v}}_{_1}-\frac{1}{3}\mathring{\bm{v}}_{_2} ,\quad
				\bm{d}_6:=-\frac{1}{3}\mathring{\bm{v}}_{_2} .
			\end{aligned}
		\end{gather*}
		To determine the structure of the interface, we choose the periodic direction and the extending direction $\bm{v}_{\alpha},\bm{v}_{\beta}\in\Lambda$ so that the following relation holds:
		\[\Lambda = \{m\bm{v}_{\alpha}+n\bm{v}_{\beta}:(m,n)\in\mathbb{Z}^2\}.\] 
		The interface is produced by choosing 
		\[
		\Lambda^{(i)}=\{m\bm{v}_{\alpha}+n\bm{v}_{\beta}:n\ge 0,\ m\in\mathbb{Z}\},\quad 1\le i\le 6,
		\]
		and
		\[
		\Lambda^{(i)}=\{m\bm{v}_{\alpha}+n\bm{v}_{\beta}:n<0,\ m\in\mathbb{Z}\},\quad 7\le i\le 12.
		\]
		This creates two bulk media separated by a straight interface parallel to $\mathring{\bm{v}}_{\alpha}$. 
		
		We illustrate two examples in \Cref{fig:gen_honey}. 
		\begin{figure}[htbp] 
			\centering
			\includegraphics[width = 0.45\textwidth]{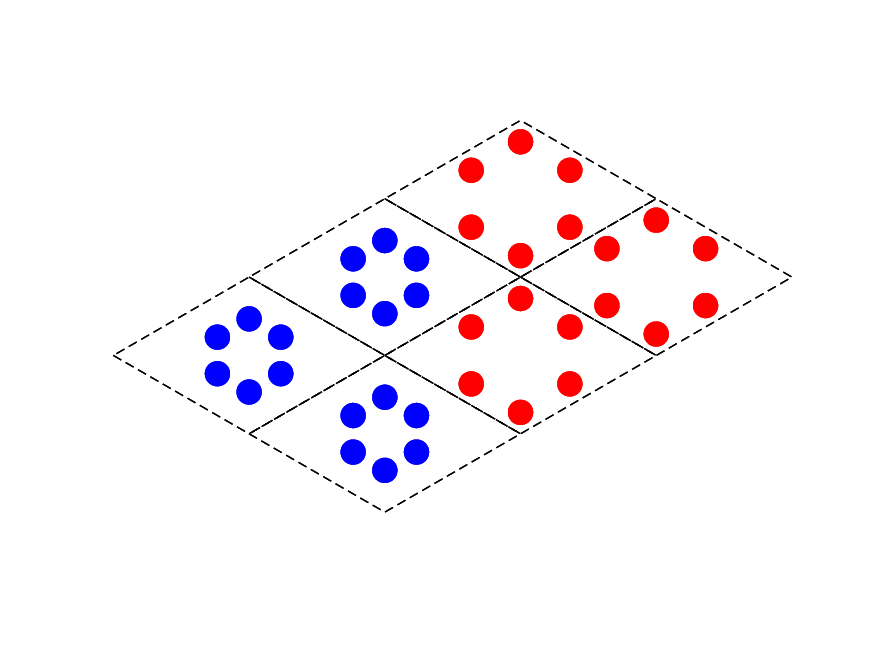}
			\hspace{-0.4cm}
			\includegraphics[width = 0.45\textwidth]{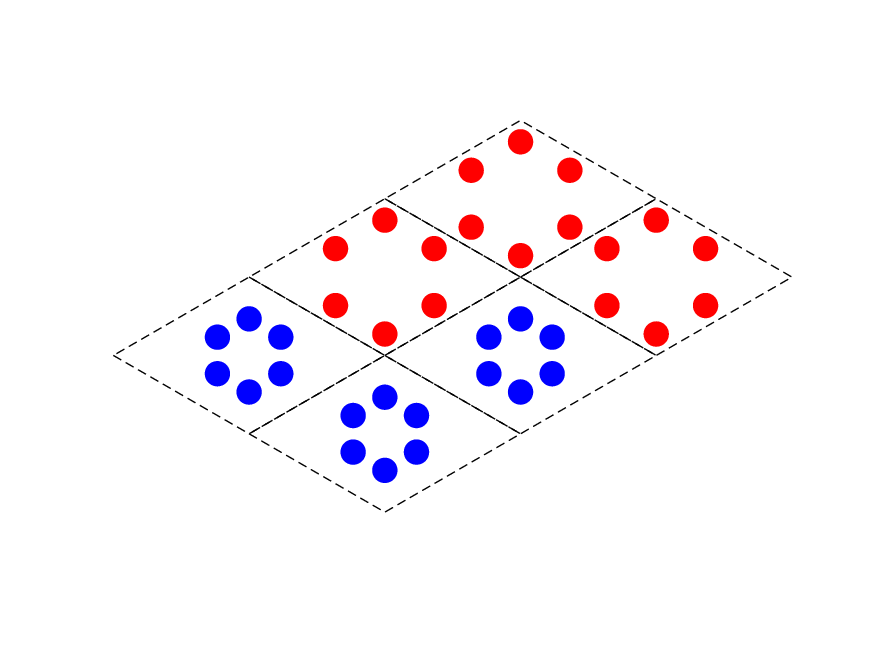}
			\caption{Interfaces in a generalized honeycomb structure. Left panel: type-I interface. Right panel: type-II interface.}
			\label{fig:gen_honey}
		\end{figure}
		In the left panel of the figure, we let 
		\[ \bm{v}_{\alpha} = \mathring{\bm{v}}_{_2}-\mathring{\bm{v}}_{_1},\quad \bm{v}_{\beta} = \mathring{\bm{v}}_{_2} \]
		and, in the right panel, we let 
		\[ \bm{v}_{\alpha} = \mathring{\bm{v}}_{_2}+\mathring{\bm{v}}_{_1},\quad \bm{v}_{\beta} = \mathring{\bm{v}}_{_2}. \]
		In \cite{borui-cmp}, these interfaces are referred to as \emph{type-I} and \emph{type-II} interfaces, respectively. 
		
	\end{example}

	\subsection{Subwavelength spectral problem and equivalent formulations}
	We let $\delta>0$ be the material contrast parameter and assume that $\delta \ll 1$. We set $\Omega:=\mathbb{R}^2\setminus \overline{\mathcal{D}}$ and consider the following subwavelength spectral problem for $z=\mathcal{O}(\delta)$:
	\begin{equation} \label{eq_full_pde_subwavelength_resonance}
		\left\{
		\begin{aligned}
			&-\Delta u-z u=0 \quad \text{in } \Omega:=\mathbb{R}^2\setminus \overline{\mathcal{D}}, \\
			&-\Delta u-z u=0 \quad \text{in } \mathcal{D}, \\
			&u\big|_{-}=u\big|_{+} \quad \text{on } \partial \mathcal{D}, \\
			&\left.\frac{\partial u}{\partial \nu}\right|_{-}=\left.\delta\frac{\partial u}{\partial \nu}\right|_{+} \quad \text{on } \partial \mathcal{D}.
		\end{aligned}
		\right.
	\end{equation}
	Here and thereafter, $\nu$ denotes the unit outward normal to $\partial \mathcal{D}$  and the subscripts $+$ and $-$ denote traces taken from the exterior and interior of $\mathcal{D}$, respectively.
	
	There are two equivalent formulations of \eqref{eq_full_pde_subwavelength_resonance}. The first is to write \eqref{eq_full_pde_subwavelength_resonance} as the eigenvalue problem of an elliptic operator:
	\begin{equation*}
		(\mathcal{L}_{\delta}-z)u=0
	\end{equation*}
	in which $\mathcal{L}_{\delta}$ is defined as
	\begin{gather*}
		\begin{aligned}
			\mathcal{L}_{\delta}&:\operatorname{dom}(\mathcal{L}_{\delta})\subset L^2(\mathbb{R}^2)\to L^2(\mathbb{R}^2),\\ \mathcal{L}_{\delta}u&:=-\big(\delta^{-1}\mathbbm{1}_{\Omega}+\mathbbm{1}_{\mathcal{D}}\big)\nabla\cdot \Big(\big(\delta\mathbbm{1}_{\Omega}+\mathbbm{1}_{\mathcal{D}} \big)\nabla u\Big),
		\end{aligned}
	\end{gather*}
	with the domain
	\begin{equation*}
		\operatorname{dom}(\mathcal{L}_{\delta}):=\Big\{u\in H^1(\mathbb{R}^2):\big(\delta^{-1}\mathbbm{1}_{\Omega}+\mathbbm{1}_{\mathcal{D}}\big)\nabla\cdot \Big(\big(\delta\mathbbm{1}_{\Omega}+\mathbbm{1}_{\mathcal{D}} \big)\nabla u\Big)\in L^2(\mathbb{R}^2)\Big\}.
	\end{equation*}
	Here, $\mathbbm{1}_{\Omega}$ (respectively, $\mathbbm{1}_{\mathcal{D}}$) denotes the characteristic function of $\Omega$ (respectively, $\mathcal{D}$). In the following, we use $H^{1}(\mathcal{D})$ for the space of functions in $L^2(\mathcal{D})$ whose weak derivatives also  lie in $L^2(\mathcal{D})$, $H^{1/2}(\partial \mathcal{D})$ for the trace space of $H^1(\mathcal{D})$ and $H^{-1/2}(\partial \mathcal{D})$ for the dual of $H^{1/2}(\partial \mathcal{D})$.  
	
	Since we are interested in the spectral property of $\mathcal{L}_{\delta}$ as $\delta\to 0$, whose $\delta$-dependent domain hinders the asymptotic analysis, it is convenient to introduce an alternative characterization of the spectral problem \eqref{eq_full_pde_subwavelength_resonance}, which has a $\delta$-independent domain. The idea is to reformulate the spectral problem \eqref{eq_full_pde_subwavelength_resonance} in the interior domain $\mathcal{D}$ (\emph{i.e.}, inside the resonators) by eliminating the transmission boundary conditions with the help of the exterior Dirichlet-to-Neumann map \cite{Feppon2024}. In this case, the exterior Dirichlet-to-Neumann map $\mathcal{T}^{z}$ is defined as follows. 
	\begin{definition}
		For $\phi\in H^{1/2}(\partial \mathcal{D})$, the exterior Dirichlet-to-Neumann map $\mathcal{T}^{z}$ is given by
		\begin{equation*}
			\begin{aligned}
				\mathcal{T}^{z}:\quad
				H^{1/2}(\partial \mathcal{D}) &\to H^{-1/2}(\partial \mathcal{D}), \\
				\phi & \mapsto \left.\frac{\partial u}{\partial \nu}\right|_{+},
			\end{aligned}
		\end{equation*}
		where $u$ is the unique solution to the following equation that is posed outside the set $\mathcal{D}$ of resonators:
		\begin{equation} \label{eq_d2n_def}
			\left\{
			\begin{aligned}
				&-\Delta u-z u=0 \quad \text{in } \Omega, \\
				&u\big|_{+}=\phi \quad \text{on } \partial \Omega .
			\end{aligned}
			\right.
		\end{equation}
	\end{definition} 
	Following this definition, we note that \eqref{eq_full_pde_subwavelength_resonance} is equivalent to the following problem:
	\begin{equation} \label{eq_interior_pde_subwavelength_resonance}
		\left\{
		\begin{aligned}
			&-\Delta u-z u=0 \quad \text{in } \mathcal{D}, \\
			&\left.\frac{\partial u}{\partial \nu}\right|_{-}=\delta\mathcal{T}^{z}\big[u|_{\partial \mathcal{D}} \big] \quad \text{on } \partial \mathcal{D}.
		\end{aligned}
		\right.
	\end{equation}
	
	The basic properties of the map $\mathcal{T}^{z}$ are summarized below. They were proved for the case where $\mathcal{D}$ contains finitely many resonators in \cite[Proposition 3.1]{Feppon2024}, and for the periodic case in \cite[Proposition 2.2]{qiu2025nonlinear}. We now extend them to arbitrary configurations $\mathcal{D}$ satisfying \Cref{asmp_geometry_assumption}, and refer to  \Cref{app_d2n} for their proofs.
	\begin{proposition} \label{prop_d2n_map}
		There exists $z_{0}>0$, depending only on $\mathcal{D}$, such that $\mathcal{T}^{z}$ is well-defined on the disk $\mathbb{D}_0:=\{z\in\mathbb{C}:\, |z|< z_0\}$. 
		In addition, the operator-valued function
		$$
		z\mapsto \mathcal{T}^{z}\in \mathcal{B}(H^{1/2}(\partial \mathcal{D}),H^{-1/2}(\partial \mathcal{D}))
		$$
		is analytic in $\mathbb{D}_0$ and continuous on $\overline{\mathbb{D}_0}$.
	\end{proposition}
	Furthermore, it is standard to check that \eqref{eq_interior_pde_subwavelength_resonance} admits the following variational characterization:
	\begin{equation} \label{eq_variational_characterization_subwavelength_resonance}
		\mathfrak{a}(u,v;z,\delta)=0,\quad \forall v\in H^1(\mathcal{D}),
	\end{equation}
	where the sesquilinear form $\mathfrak{a}$ is defined, for a spectral parameter $\lambda\in \mathbb{D}_0$, by
	\begin{equation} \label{eq_sesquilinear_form_subwavelength_resonance}
		\mathfrak{a}(u,v;\lambda,\delta)
		:=(\nabla u,\nabla v)_{(L^2(\mathcal{D}))^2} -\lambda (u,v)_{L^2(\mathcal{D})} -\delta (\mathcal{T}^{\lambda}[u],v)_{L^2(\partial \mathcal{D})} .
	\end{equation}
	
	We employ \eqref{eq_variational_characterization_subwavelength_resonance} to analyze the asymptotic behavior of the subwavelength resonance as $\delta\to 0$. The analysis is based on the fact that the relevant spectral information is encoded in the resolvent. For a spectral parameter $\lambda\in \mathbb{D}_0$, we denote by $\mathcal{R}(\lambda,\delta)$ the resolvent of \eqref{eq_variational_characterization_subwavelength_resonance}, whenever it exists, namely the bounded map
	\begin{equation*}
		\mathcal{R}=\mathcal{R}(\lambda,\delta):L^2(\mathcal{D})\to H^1(\mathcal{D})
	\end{equation*}
	such that, for $f\in L^2(\mathcal{D})$, we have
	\begin{equation*}
		\mathfrak{a}(\mathcal{R}f,v;\lambda,\delta)=(f,v)_{L^2(\mathcal{D})}, \quad \forall v\in H^1(\mathcal{D}).
	\end{equation*}
	In the subwavelength regime studied below, we specialize this notation to $\lambda=\delta z$ with $z$ ranging in a fixed compact set $\mathbb{K}\subset \mathbb{C}\setminus \mathbb{R}$. \Cref{app_inverse_of_form} proves that this resolvent is well-defined in that regime for sufficiently small $\delta$. 
	
	\subsection{Convergence of the resolvent to the capacitance operator}
	As discussed in the introduction, the subwavelength resonance is naturally linked to a discrete problem governed by the capacitance operator $\mathcal{C}$. When the resonators in each polygon $Y^{(m)}+\Sigma^{(m)}$ are numbered from $1$ to $K$, the capacitance operator 
	$\mathcal{C}:\ell^2(\Sigma)\to \ell^2(\Sigma)$ is defined by 
	\begin{equation} \label{eq_cap_operator_def}
		(\mathcal C \boldsymbol{\Phi})(\sigma'):=\sum_{\sigma\in\Sigma}\mathcal C(\sigma,\sigma')\Phi(\sigma),\quad
		\mathcal{C}(\sigma,\sigma'):=-\int_{\partial D_{\sigma'}}\mathcal{T}^{0}[\mathbbm{1}_{\partial D_{\sigma}}]\,\du \sigma(y),
	\end{equation}
	for $\boldsymbol{\Phi}=\{\Phi(\sigma)\}_{\sigma\in\Sigma}\in\ell^2(\Sigma)$. 
	The operator $\mathcal{C}$ is bounded, following from the exponential decay of its off-diagonal entries; see \Cref{apthm:exp_decay}.
	
	One of our main theoretical results in this paper is to rigorously establish the norm convergence (as $\delta \rightarrow 0$) of the subwavelength resolvent $\mathcal{R}(\delta z,\delta)$ associated with the continuous problem to the one associated with the discrete problem in the following sense.
	\begin{theorem} \label{thm_cont_to_disc_resolvent_converge}
		Let $\mathbb{K}\subset \mathbb{C}$ be a compact region such that $\mathbb{K}\cap \mathbb{R}=\emptyset$. Then, there exists $C,\delta_1>0$ depending only on the radius of $\mathbb{K}$ such that, for any $z\in \mathbb{K}$ and
		\begin{equation*}
			\delta<\delta_1\min\{1,\frac{1}{\|(\mathcal{C}-z)^{-1}\|_{\mathcal{B}(\ell^2(\Sigma))}}\},
		\end{equation*}
		it holds that
		\begin{equation} \label{eq_cont_to_disc_resolvent_converge}
			\big\|\delta \mathcal{R}(\delta z,\delta) - P^{\ast} (\mathcal{C}-z)^{-1} P\big\|_{\mathcal{B}(L^2(\mathcal{D}))} \le  C\delta \|(\mathcal{C}-z)^{-1}\|_{\mathcal{B}(\ell^2(\Sigma))}^2,
		\end{equation}		
		where $P:L^2(\mathcal{D})\to \ell^2(\Sigma)$ is the projection to the piecewise-constant component on each resonator
		\begin{equation} \label{eq_resonator_projection}
			(Pu)(\sigma):=\int_{D_{\sigma}}u\,\du x,\quad \forall \sigma\in \Sigma .
		\end{equation}
		Here and throughout, the superscript $\ast$ denotes the adjoint and $\mathcal{B}(X,Y)$ denotes the space of bounded linear operators from $X$ to $Y$. We write $\mathcal{B}(X):=\mathcal{B}(X,X)$.
	\end{theorem}
	
	Some remarks are now in order. It is worth emphasizing that the result of Theorem \ref{thm_cont_to_disc_resolvent_converge} is the analog for classical waves of the tight-binding approximation of Schr\"odinger operators in the strong-field regime that has been rigorously established in \cite{Shapiro2022,Fefferman2018}, in the same resolvent-convergence framework; see also \cite{Dimassi1999,Nenciu1991,Panati2003,Ablowitz2020,fiorani} and the references therein. Besides the physical distinction that we work here with a classical wave system, the most remarkable feature of Theorem \ref{thm_cont_to_disc_resolvent_converge}, from both physical and mathematical perspectives, is that \textit{the resulting discrete operator $\mathcal{C}$ encodes interactions between resonators at arbitrary distances, rather than being restricted to the nearest-neighbor case}. The incorporation of arbitrary-distance interactions makes it possible to realize phenomena that cannot arise in systems with only nearest-neighbor interactions. For example, the Haldane model, which requires next-nearest-neighbor interactions, realizes the photonic Hall effect \cite{Haldane2008}.
	
	From a technical perspective, our resonator structure leads to a much simpler proof of resolvent convergence than the one in \cite{Shapiro2022,Fefferman2018}. This is because the eigenbasis corresponding to the subwavelength spectrum, namely the constant functions on each resonator, is naturally orthogonal due to the fact that the resonators are disjoint. In contrast, the eigenmodes associated with the low-energy spectrum of Schr\"odinger operators are expected to have exponentially small overlap, which requires a delicate analysis as in \cite{Shapiro2022,Fefferman2018}. To this end, the potential functions in \cite{Shapiro2022} are assumed to be radially symmetric, which guarantees the exponentially small overlap mentioned above; such a geometric assumption is not required for Theorem \ref{thm_cont_to_disc_resolvent_converge}. Nevertheless, we point out that the main idea of the proof is similar to that in \cite{Shapiro2022} and indeed inspired our proof of Theorem \ref{thm_cont_to_disc_resolvent_converge}. In fact, the key mechanism underlying the resolvent convergence is based on the following two facts:
	\begin{enumerate}
		\item[(i)] There exists a uniform separation between the low-lying spectrum of interest and the high-energy counterpart;
		\item[(ii)] The spectral projection of the low-lying spectrum admits a clear characterization.
	\end{enumerate}
	
	Finally, we point out that Theorem \ref{thm_cont_to_disc_resolvent_converge} generalizes the results in \cite{Ammari2023} on the discrete approximation of defect modes in the subwavelength regime, as illustrated in Example \ref{exmp_resolvent_converge_periodic}; see also \cite{Ammari2024b}. Moreover, the approximation of interface modes was studied numerically in the aforementioned works and is now rigorously justified by Theorem \ref{thm_cont_to_disc_resolvent_converge}, as will be discussed in Example \ref{exmp_resolvent_converge_interface}. We also note that the techniques in \cite{Ammari2023} require periodicity of the system, which is not needed in the proof of Theorem \ref{thm_cont_to_disc_resolvent_converge}.
	
	We outline the main idea of the proof of Theorem \ref{thm_cont_to_disc_resolvent_converge} in the next section and refer to \Cref{sec_resolvent_convergence} for details.

	\section{Main idea of the proof of Theorem \ref{thm_cont_to_disc_resolvent_converge} and its applications}
	\label{sec_heuristics}
	\subsection{Main idea of the proof of \Cref{thm_cont_to_disc_resolvent_converge}}
	Physically, the resolvent convergence shown in Theorem \ref{thm_cont_to_disc_resolvent_converge} is rooted in the mechanism that, in the high-contrast regime $\delta\ll 1$, the wave in \eqref{eq_full_pde_subwavelength_resonance} (that is, the solution \eqref{eq_full_pde_subwavelength_resonance}) with a frequency of order $\sqrt{\delta}$ is confined inside the resonator. As a consequence, the correlation between waves in different resonators is controlled by the Dirichlet-to-Neumann map. The leading-order term of the latter is of the monopole type, equaling the product of the average value of the wave inside the resonator and the point-to-point approximation of the Dirichlet-to-Neumann map, \emph{i.e.}, the kernel of the capacitance operator $\mathcal{C}$.
	
	Based on this intuition, the proof of Theorem \ref{thm_cont_to_disc_resolvent_converge} is conducted on two orthogonal components of $L^2(\mathcal{D})$: the functions that are constant in each connected component $\bm{u} + D^{(n)}$, corresponding to the wave in the subwavelength regime and its orthogonal complement. To be more precise, we recall that $ \mathcal{D} = \bigcup_{\sigma\in \Sigma} D_{\sigma}$ and introduce the following projections:
	\begin{equation*}
		\mathcal{P}:L^2(\mathcal{D})\to L^2(\mathcal{D}),\quad \mathcal{P} u:=\sum_{\sigma\in\Sigma}\Big(\int_{D_{\sigma}}u \, \du x\Big)\mathbbm{1}_{D_{\sigma}},\quad\quad
		\mathcal{P}_{\perp}:=\operatorname{Id}-\mathcal{P} .
	\end{equation*}
	With the normalization $|D_\sigma|=1$, the projection $P$ in \eqref{eq_resonator_projection} and the orthogonal projection $\mathcal P$ are related by $\mathcal P=P^\ast P$; equivalently, $P$ records the constant-mode coefficients and $\mathcal P$ lifts them back to piecewise-constant functions on $\mathcal D$. Then Theorem \ref{thm_cont_to_disc_resolvent_converge} follows from the following convergence: for $z\in \mathbb{K}$,
	\begin{equation} \label{eq_strong_resolvent_con_perp_to_perp}
		\delta \mathcal{P}_{\perp}\mathcal{R}(\delta z,\delta)\mathcal{P}_{\perp} \to 0,
	\end{equation}
	\begin{equation} \label{eq_strong_resolvent_con_perp_to_Pi}
		\delta \mathcal{P}_{\perp}\mathcal{R}(\delta z,\delta)\mathcal{P} \to 0,\quad
		\delta \mathcal{P}\mathcal{R}(\delta z,\delta)\mathcal{P}_{\perp} \to 0,
	\end{equation}
	and
	\begin{equation} \label{eq_strong_resolvent_con_Pi_to_Pi}
		\mathcal{P}\Big(\delta \mathcal{R}(\delta z,\delta) - P^{\ast} (\mathcal{C}-z)^{-1} P\Big)\mathcal{P} \to 0,
	\end{equation}
	in the operator norm. Here, we illustrate the main idea of the diagonal convergence \eqref{eq_strong_resolvent_con_perp_to_perp} and \eqref{eq_strong_resolvent_con_Pi_to_Pi}, which demonstrate the fundamental mechanism underlying Theorem \ref{thm_cont_to_disc_resolvent_converge}.
	
	First, we consider \eqref{eq_strong_resolvent_con_perp_to_perp}. The key idea underlying this convergence is that, since $\mathcal{P}_{\perp}$ is `approximately' the spectral projection to the high-frequency spectrum associated with the problem \eqref{eq_full_pde_subwavelength_resonance}, which is isolated from the subwavelength spectrum by an $\mathcal{O}(1)$ distance (see Figure \ref{fig_decomposition_spectrum_sub_vs_high} for an illustration), the resolvent $\mathcal{P}_{\perp}\mathcal{R}(\delta z,\delta)\mathcal{P}_{\perp}$ is uniformly bounded as $\delta\to 0$:
	\begin{equation} \label{eq_strong_resolvent_con_main_idea_1}
		\|\mathcal{P}_{\perp}\mathcal{R}(\delta z,\delta)\mathcal{P}_{\perp}\|_{\mathcal{B}(L^2(\mathcal{D}))}\lesssim 1,
	\end{equation}
	which implies \eqref{eq_strong_resolvent_con_perp_to_perp}. Here, the notation $A \lesssim B$ means that there exists a constant $C$ independent of $\delta$ such that $A \leq C B$.
	
	The key idea to justify \eqref{eq_strong_resolvent_con_main_idea_1} lies in the following coercivity estimate:
	\begin{equation} \label{eq_strong_resolvent_con_main_idea_2}
		\mathfrak{a}(\mathcal{P}_{\perp}u,\mathcal{P}_{\perp}u;\delta z,\delta)\geq c_{*}\|\mathcal{P}_{\perp}u\|_{H^1(\mathcal{D})}^2,\quad \forall u\in H^1(\mathcal{D})
	\end{equation}
	for small $\delta$, with $c_*>0$ independent of $\delta$. The validity of \eqref{eq_strong_resolvent_con_main_idea_2} can be seen from the definition of the sesquilinear form $\mathfrak{a}$ in \eqref{eq_sesquilinear_form_subwavelength_resonance} by the Poincaré–Wirtinger inequality; we provide a detailed proof in \Cref{sec_regularity_estimate}. Nevertheless, we note that the above heuristics are not rigorous since the Dirichlet-to-Neumann map does not commute with the projection $\mathcal{P}$ in general, that is, 
	\begin{equation} \label{eq_strong_resolvent_con_main_idea_7}
		(\mathcal{T}^{\delta z}\mathcal{P} u,\mathcal{P}_{\perp} u)_{L^2(\partial \mathcal{D})}\neq 0,
	\end{equation}
	which prevents a direct implication of \eqref{eq_strong_resolvent_con_main_idea_1} from \eqref{eq_strong_resolvent_con_main_idea_2}. We solve this problem by a Lyapunov-Schmidt-type argument in \Cref{sec_resolvent_convergence_proof} (Step 1). That is, we bound $\|\mathcal{P} u\|$ emerging from \eqref{eq_strong_resolvent_con_main_idea_7} in terms of $\|\mathcal{P}_{\perp} u\|$ using the resolvent of the capacitance operator, which then leads to a closed estimate in $\|\mathcal{P}_{\perp} u\|$ and guarantees \eqref{eq_strong_resolvent_con_main_idea_1}. See \Cref{sec_resolvent_convergence_proof} for more details.
	
	\begin{figure}[htbp]
		\centering
		\begin{tikzpicture}[line join=round, line cap=round, thick, scale=0.6]
			\draw (0,0)--(16,0);
			\draw[line width=0.7mm] (0,-0.5)--(0,0.5);
			\node[left,scale=0.8] at (0,0) {$0$};
			
			\draw[line width=0.9mm, red] (0,0)--(3,0);
			\draw[decorate,decoration={brace,mirror}] (0,-1) -- (3,-1);
			\draw[decorate,decoration={brace}] (0,1) -- (3,1);
			\node[above,scale=0.8] at (1.5,1.1) {$\mathcal{O}(\delta)$};
			\draw[decorate,decoration={brace,mirror}] (0,-1) -- (3,-1);
			\node[below,scale=0.8] at (1.5,-1.5) {subwavelength spectrum};
			
			\draw[decorate,decoration={brace}] (0,2) -- (8,2);
			\node[above,scale=0.8] at (4,2.1) {$\mathcal{O}(1)$};
			\draw[line width=0.9mm, blue] (8,0)--(16,0);
			\draw[decorate,decoration={brace,mirror}] (8,-1) -- (16,-1);
			\node[below,scale=0.8] at (12,-1.5) {high-frequency spectrum};
		\end{tikzpicture}
		\caption{Distribution of the spectrum associated with the problem \eqref{eq_full_pde_subwavelength_resonance}. The red region refers to the subwavelength spectrum with $z=\mathcal{O}(\delta)$, with the associated spectral projection being `approximately' $\mathcal{P}$, while the blue region represents the high-frequency spectrum.}
		\label{fig_decomposition_spectrum_sub_vs_high}
	\end{figure}
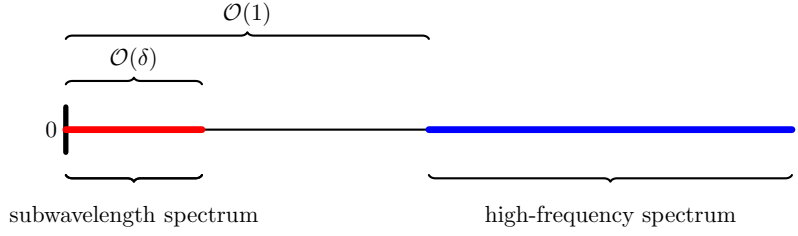
	
	Next, we consider \eqref{eq_strong_resolvent_con_Pi_to_Pi}. We need to prove that the following convergence holds uniformly in the norm of $f$:
	\begin{equation} \label{eq_strong_resolvent_con_main_idea_6}
		\delta\mathcal{P} u-\mathcal{P} P^{\ast} (\mathcal{C}-z)^{-1}Pf\to 0,
	\end{equation}
	where $f\in \operatorname{Ran}\mathcal{P}$ and $u$ solves
	\begin{equation} \label{eq_strong_resolvent_con_main_idea_4}
		\mathfrak{a}(u,v;\delta z,\delta)=(\nabla u,\nabla v)_{(L^2(\mathcal{D}))^2} -\delta z(u,v)_{L^2(\mathcal{D})} -\delta (\mathcal{T}^{\delta z}[u],v)_{L^2(\partial \mathcal{D})}= (f,v)_{L^2(\mathcal{D})}.
	\end{equation}
	The function $u$ is expected to be approximately constant inside each resonator, just as the source $f$. Hence, taking $v=\mathbbm{1}_{D_{\sigma}}$ in \eqref{eq_strong_resolvent_con_main_idea_4}, replacing $u$ with $\mathcal{P} u$, and passing formally from $\mathcal{T}^{\delta z}$ to $\mathcal{T}^{0}$, we see that
	\begin{equation*}
		-\delta z\int_{D_{\sigma}} u\, \du x+\sum_{\sigma'\in\Sigma}\mathcal{C}(\sigma,\sigma')\big(\int_{D_{\sigma'}} u\, \du x\big)\approx Pf(\sigma),\quad \forall \sigma\in \Sigma .
	\end{equation*}
	Equivalently,
	\begin{equation*}
		\delta(\mathcal{C}-z)Pu\approx Pf.
	\end{equation*}
	Since $z\notin \mathbb{R}$, the invertibility of $\mathcal{C}-z$ then implies that $\delta Pu\approx (\mathcal{C}-z)^{-1}Pf$. This illustrates the validity of \eqref{eq_strong_resolvent_con_main_idea_6}. All these approximate arguments are rigorously formulated in \Cref{sec_resolvent_convergence_proof}.
	
	\subsection{Implication of \Cref{thm_cont_to_disc_resolvent_converge}}
	The resolvent convergence in Theorem \ref{thm_cont_to_disc_resolvent_converge} rigorously justifies the discrete approximation of defect modes and guided modes in a straight and bent waveguides. 
	In both cases, a key feature is that the resolvent of $\mathcal{C}$ can be obtained via the limiting absorption principle. This is captured by the following condition.
	\begin{definition} \label{asmp_resolvent_boundedness}
		The interval $\mathcal{I}\subset \mathbb{R}$ is said to be $\mathcal{C}$-admissible if the following conditions hold:
		\begin{enumerate}
			\item[(i)] for any $\epsilon_0>0$,
			\begin{equation*}
				\sup_{z:\Re z\in \mathcal{I},0<\Im z <\epsilon_0}\|(\mathcal{C}-z)^{-1}\|_{\mathcal{B}(\ell^1(\Sigma), \ell^{\infty}(\Sigma))}<\infty;
			\end{equation*}
			\item[(ii)] for $z\in \mathcal{I}$, $(\mathcal{C}-z-\iu\epsilon)^{-1}\in \mathcal{B}(\ell^1(\Sigma), \ell^{\infty}(\Sigma))$ converges in the operator norm as $\epsilon\to 0^+$:
			\begin{equation*}
				(\mathcal{C}-z-\iu0^+)^{-1}:=
				\lim_{\epsilon\to 0^+}(\mathcal{C}-z-\iu\epsilon)^{-1} \in \mathcal{B}(\ell^1(\Sigma), \ell^{\infty}(\Sigma)).
			\end{equation*}
		\end{enumerate}
	\end{definition}
	
	\begin{example}[Approximation of defect modes] \label{exmp_resolvent_converge_periodic}
		We consider a more general setting for defect modes. Let $\Lambda = \mathbb{Z}^2$ and consider materials with compact defects, i.e., there exists $K_{\rm def}>0$ such that
		\begin{equation*}
			\left\{\begin{aligned}
				\Lambda^{(1)}&=\mathbb{Z}^2\backslash \{-K_{\rm def},-K_{\rm def}+1,\ldots,K_{\rm def}\}^2,\\
				|\Lambda^{(i)}|&<\infty,\quad i\in \{2,3,\ldots,N_{\mathrm{type}}\},
			\end{aligned}\right.			
		\end{equation*}
		Then, $\mathcal{C}$ may possess an in-gap defect mode, as illustrated in Figure \ref{fig_defect_eigenvalue}. In that case, the intervals $\mathcal{I}_1$ and $\mathcal{I}_2$, which separate the in-gap eigenvalue from the essential spectrum, are admissible in the sense of Definition \ref{asmp_resolvent_boundedness}. As a consequence, Theorem \ref{thm_cont_to_disc_resolvent_converge} applies to all $z\in\Gamma$, where $\Gamma$ is a closed contour enclosing the defect eigenvalue. By the Riesz projection, the following convergence holds in $W^{1,\infty}(\mathcal{D})$ for any $f\in W^{1,1}(\mathcal{D})$ as $\delta\to 0$:
		\begin{align*}
			&\delta\mathbbm{1}_{\{\delta\lambda_{\text{defect}}\}}(\mathcal{L}_{\delta})f \\
			&\quad= -\frac{1}{2\pi \iu}\oint_{\Gamma}\delta\mathcal{R}(\delta z,\delta)dz
			\to -\frac{1}{2\pi \iu}\oint_{\Gamma}P^{\ast}(\mathcal{C}-z)^{-1}Pfdz =P^{\ast} \mathbbm{1}_{\{\lambda_{\text{defect}}\}}(\mathcal{C})Pf.
		\end{align*}
		Here, $W^{1,\infty}(\mathcal{D})$ (resp.  $W^{1,1}(\mathcal{D})$)  is the set of functions with their weak derivatives in $L^\infty(\mathcal{D})$ (resp. $L^1(\mathcal{D})$), and $\Gamma$ is a contour in the rescaled z-plane; the corresponding contour for $\lambda$ is $\delta\Gamma$. This implies that if the discrete operator $\mathcal{C}$ possesses in-gap eigenmodes, then the same is true for the continuous problem \eqref{eq_full_pde_subwavelength_resonance}. In other words, Theorem \ref{thm_cont_to_disc_resolvent_converge} justifies the discrete approximation of the defect modes.
		
		\begin{figure}[htbp]
			\centering
			\begin{tikzpicture}[line join=round, line cap=round, thick, scale=0.8]
				\draw (-8,0)--(8,0);
				
				\node[above,scale=1] at (0,0.2) {$\lambda_{\text{defect}}$};
				\draw[fill=red] plot [smooth] (0,0) ellipse (0.1cm and 0.1cm);
				
				\node[scale=1.2] at (-3,0) {$)$};
				\node[scale=1.2] at (3,0) {$($};
				\draw[line width=0.7mm, blue] (-8,0)--(-3,0);
				\draw[line width=0.7mm, blue] (8,0)--(3,0);
				\draw[decorate,decoration={brace,mirror}] (-8,-1) -- (-3,-1);
				\draw[decorate,decoration={brace,mirror}] (3,-1) -- (8,-1);
				\node[below,scale=0.8] at (-5.5,-1.5) {lower-gap spectrum};
				\node[below,scale=0.8] at (5.5,-1.5) {upper-gap spectrum};
				
				\draw[fill=blue,opacity=0.3] (-2,0.2) rectangle (-1,-0.2);
				\node[below,scale=0.8] at (-2,-0.3) {$\mathcal{I}_1$};
				\draw[fill=blue,opacity=0.3] (1,0.2) rectangle (2,-0.2);
				\node[below,scale=0.8] at (2,-0.3) {$\mathcal{I}_2$};
				
				\draw[dashed,thick] plot [smooth] (0,0) ellipse (1.5cm and 1.8cm);
				\node[right,scale=1] at (1,1.5) {$\Gamma$};
			\end{tikzpicture}
			\vspace{0.2cm}
			\caption{Defect eigenvalue.}
			\label{fig_defect_eigenvalue}
		\end{figure}
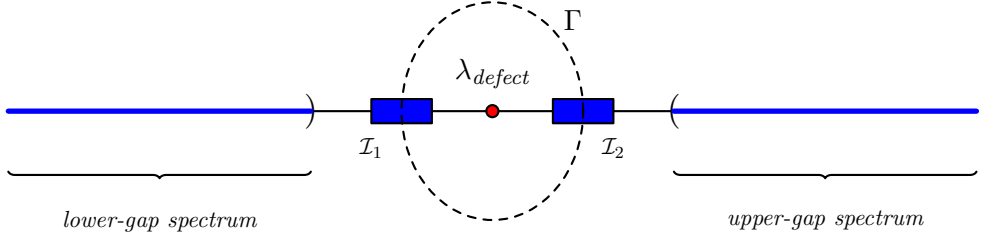
		
	\end{example}
	
	\begin{example}[Approximation of interface modes] \label{exmp_resolvent_converge_interface}
		In \cite{Ammari2026}, it is proved that $\mathcal{C}$ possesses in-gap interface modes whose eigenvalues lie in the bulk spectral gap $\mathcal{I}_{\varepsilon}$. The dispersion curves, plotted against the momentum parallel to the interface, are depicted in Figure \ref{fig_interface_eigenvalue_curve}. Importantly, following the analysis of \cite{Fliss2016}, it can be shown that the limiting absorption principle holds everywhere in the essential spectral gap \emph{except where the dispersion curve is locally flat}. To be more precise, suppose that there are two branches of the dispersion curves in the essential gap, as depicted in Figure \ref{fig_interface_eigenvalue_curve}, and that they have nonvanishing derivatives everywhere except at the local extrema located at zero parallel momentum. Denote the energy levels associated with these two extreme points by $z_{upp}$ and $z_{low}$. Then, any closed interval $\mathcal{I}\subset \mathcal{I}_{\varepsilon}$ such that $\mathcal{I}\cap \{z_{upp},z_{low}\}=\emptyset$ is $\mathcal{C}$-admissible (\emph{e.g.}, $\mathcal{I}_1$ and $\mathcal{I}_2$ in Figure \ref{fig_interface_eigenvalue_admissible}). Moreover, the limiting resolvent admits the following far-field asymptotics along the interface:
		\begin{equation} \label{eq_interface_curve_limiting_absorption_left}
			\lim_{\bm{n}\cdot \mathring{\bm{v}}_{_2}\to -\infty}\Big[\big((\mathcal{C}-z-\iu0^+)^{-1}f\big)(\bm{n})-\iu\frac{\big(f(\cdot),u_{inter}(\cdot;\kappa_{\parallel}^{left})\big)_{\ell^2(\Sigma)}}{|\alpha(\kappa_{\parallel}^{left})|}u_{inter}(\bm{n};\kappa_{\parallel}^{left}) \Big] =0
		\end{equation}
		and
		\begin{equation} \label{eq_interface_curve_limiting_absorption_right}
			\lim_{\bm{n}\cdot \mathring{\bm{v}}_{_2}\to \infty}\Big[\big((\mathcal{C}-z-\iu0^+)^{-1}f\big)(\bm{n})-\iu\frac{\big(f(\cdot),u_{inter}(\cdot;\kappa_{\parallel}^{right})\big)_{\ell^2(\Sigma)}}{|\alpha(\kappa_{\parallel}^{right})|}u_{inter}(\bm{n};\kappa_{\parallel}^{right}) \Big] =0.
		\end{equation}
		Here, $\kappa_{\parallel}^{left/right}$ are the parallel momenta at which the energy level $z$ intersects the dispersion curves; see Figure \ref{fig_interface_eigenvalue_curve}. The denominator $\alpha(\kappa_{\parallel}^{left/right})$ is the slope of the dispersion curve at the corresponding intersection point. Moreover, $u_{inter}(\cdot;\kappa_{\parallel}^{left/right})$ denotes the interface mode at that point. 
		
		Theorem \ref{thm_cont_to_disc_resolvent_converge} then implies that, for any $f\in W^{1,1}(\mathcal{D})$ and $\epsilon>0$, if we set $\delta=\mathcal{O}(\epsilon^2)$, then
		\begin{equation*}
			\|\delta \mathcal{R}(\delta (z+\iu\epsilon),\delta)f-\widetilde{\mathcal{R}}(z)f \|_{W^{1,\infty}(\mathcal{D})} =\mathcal{O}(\epsilon),
		\end{equation*}
		where $\widetilde{\mathcal{R}}(z)f:=P^{\ast}(\mathcal{C}-z-\iu0^+)^{-1}Pf$ admits asymptotics similar to those in \eqref{eq_interface_curve_limiting_absorption_left}-\eqref{eq_interface_curve_limiting_absorption_right}. This implies that the continuous problem \eqref{eq_full_pde_subwavelength_resonance} possesses approximate interface modes. In other words, Theorem \ref{thm_cont_to_disc_resolvent_converge} justifies the discrete approximation of the interface modes.
		
		\begin{figure}[htbp]
			\centering
			\subfigure[Dispersion curve ]{
				\label{fig_interface_eigenvalue_curve}
				\begin{tikzpicture}[scale=0.5]
					
					\draw[thick,->] (-8,0)--(8,0);
					\draw[thick] (0,0.3)--(0,-0.3);
					\node[below] at (0,0) {$0$};
					\node[right] at (8.2,0) {$\kappa_{\parallel}$};
					
					\draw[white,line width=0pt, name path=one] plot [smooth] coordinates {(-8,9) (-6,8.8) (-2,8.2) (0,8) (2,8.2) (6,8.8) (8,9)};
					\draw[white,line width=0pt, name path=two] plot [smooth] coordinates {(-8,3) (-6,3.2) (-2,3.8) (0,4) (2,3.8) (6,3.2) (8,3)};
					\draw[white,line width=0pt, name path=three] (-8,9)--(8,9);
					\draw[white,line width=0pt, name path=four] (-8,3)--(8,3);
					\tikzfillbetween[
					of=one and three,split
					] {pattern=north west lines};
					\tikzfillbetween[
					of=two and four,split
					] {pattern=north west lines};
					\fill[pattern=north west lines] (-8,12) rectangle (8,9);
					\fill[pattern=north west lines] (-8,3) rectangle (8,1);
					\draw[red,line width=2pt] plot [smooth] coordinates {(-4,3.8) (-2.5,4) (-1.5,4.5) (0,5.5) (1.5,4.5) (2.5,4) (4,3.8)};
					\draw[red,line width=2pt] plot [smooth] coordinates {(4,8.4) (2.5,8) (1.5,7.4) (0,6.5) (-1.5,7.4) (-2.5,8) (-4,8.4)};
					\draw[dashed] (-8,6.5)--(8,6.5);
					\node[left] at (-8,6.5) {$z_{upp}$};
					\draw[dashed] (-8,5.5)--(8,5.5);
					\node[left] at (-8,5.5) {$z_{low}$};
					\draw[dashed] (-8,7.5)--(8,7.5);
					\node[left] at (-8,7.5) {$z$};
					\node[below,scale=0.8] at (-2.3,7.7) {$(\kappa_{\parallel}^{left},z)$};
					\draw[fill=blue] plot [smooth] (-1.7,7.5) ellipse (0.2cm and 0.2cm);
					\node[below,scale=0.8] at (2.5,7.7) {$(\kappa_{\parallel}^{right},z)$};
					\draw[fill=blue] plot [smooth] (1.7,7.5) ellipse (0.2cm and 0.2cm);
				\end{tikzpicture}
			}
			\subfigure[Distribution of the spectrum and admissible sets]{
				\label{fig_interface_eigenvalue_admissible}
				\begin{tikzpicture}[line join=round, line cap=round, thick, scale=0.6]
					\draw (-8,0)--(8,0);
					
					\draw[line width=0.9mm, red] (3,0)--(1,0);
					\draw[line width=0.9mm, red] (-3,0)--(-1,0);
					\draw[decorate,decoration={brace}] (1,0.5) -- (3,0.5);
					\draw[decorate,decoration={brace}] (-3,0.5) -- (-1,0.5);
					\draw[decorate,decoration={brace}] (-2,1) -- (2,1);
					\node[above,scale=0.8] at (0,1.5) {interface eigenvalue};
					\draw[decorate,decoration={brace,mirror}] (-3,-1) -- (3,-1);
					\node[below,scale=0.8] at (0,-1.5) {bulk spectral gap $\mathcal{I}_{\varepsilon}$};
					
					\node[scale=1.2] at (-3,0) {$)$};
					\node[scale=1.2] at (3,0) {$($};
					\draw[line width=0.7mm, blue] (-8,0)--(-3,0);
					\draw[line width=0.7mm, blue] (8,0)--(3,0);
					\draw[decorate,decoration={brace,mirror}] (-8,-1) -- (-3,-1);
					\draw[decorate,decoration={brace,mirror}] (3,-1) -- (8,-1);
					\node[below,scale=0.8] at (-5.5,-1.5) {lower-gap spectrum};
					\node[below,scale=0.8] at (5.5,-1.5) {upper-gap spectrum};
					
					\draw[fill=blue,opacity=0.3] (-2.5,0.2) rectangle (-1.5,-0.2);
					\node[below,scale=0.8] at (-2,-0.3) {$\mathcal{I}_1$};
					\draw[fill=blue,opacity=0.3] (1.5,0.2) rectangle (2.5,-0.2);
					\node[below,scale=0.8] at (2,-0.3) {$\mathcal{I}_2$};
				\end{tikzpicture}
			}
			\caption{(a) Dispersion curves associated with the interface modes. The shaded region represents the bulk spectrum. The limiting absorption principle applies everywhere in the bulk spectral gap except at the energy levels where the dispersion curves of the interface eigenvalues are locally flat ($z_{upp}$ and $z_{low}$ in the figure). (b) Projection of the spectrum shown in (a) onto the real line: the blue region denotes the bulk spectrum and the red region denotes the in-gap interface eigenvalues. The admissible sets $\mathcal{I}_1$ and $\mathcal{I}_2$ denote the spectral regions where the limiting absorption principle applies.}
			\label{fig_interface_eigenvalue+admissible_set}
		\end{figure}
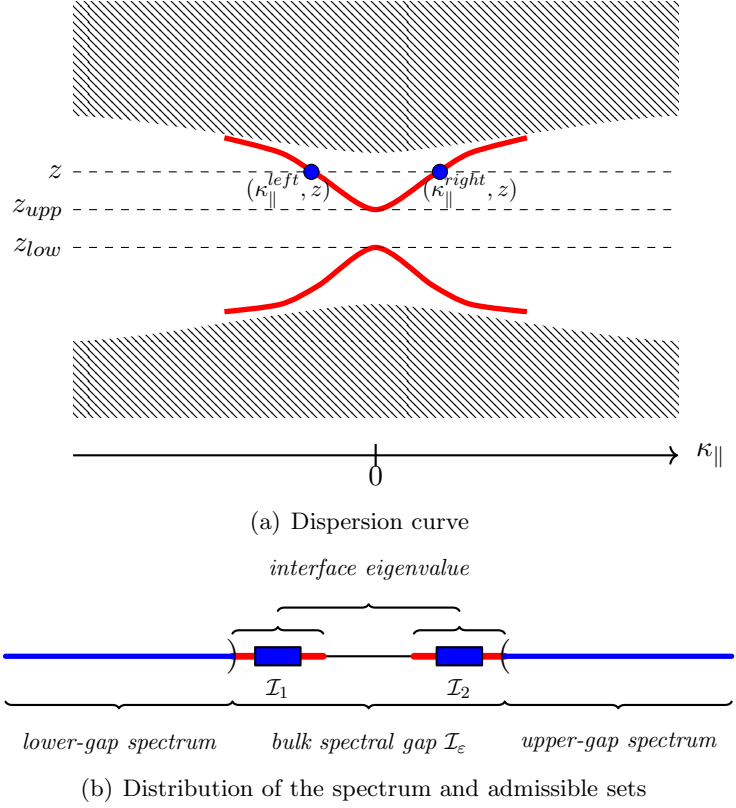
	\end{example}
	
	\begin{remark}
		The idea in Example \ref{exmp_resolvent_converge_interface} also works for the curved-interface model, except that (i) \eqref{asmp_resolvent_boundedness} may not be easy to verify and (ii) the explicit asymptotics \eqref{eq_interface_curve_limiting_absorption_left}-\eqref{eq_interface_curve_limiting_absorption_right} may not be available because of the loss of periodicity along the interface.
	\end{remark}
	
	\section{Patch approximation of the capacitance operator}\label{sec:PatchApprox}
	To efficiently calculate the spectrum of the capacitance operator $\mathcal{C}$, we construct and justify its patch approximation in this section.
	\subsection{Domain decomposition by patches}
	
	We first decompose the domain $\Omega$ into translates of finitely many patches. To do so, we introduce some definitions.
	\begin{definition}
		A \emph{patch} of the domain $\Omega$ is a bounded connected set $\mathsf{P}\subset \Omega$ of the form
		\begin{equation}
			\mathsf{P} := \bigcup_{\bm{v}\in\Pi}\big( Y_{\bm{v}}\setminus \overline{\mathcal{D}} \big),
		\end{equation}
		where $\Pi \subset \Lambda$ is a finite index set. This use of $\mathsf{P}$ is unrelated to the projection operator $P$.
		For a given $\bm{v}_{_0}\in \Lambda$, the \emph{regularized patch} of size $M$ is defined by
		\begin{equation}
			\widetilde{\mathsf{P}}_{\bm{v}_{_0}}(M) := \bigcup_{\bm{v}\in \widetilde{\Pi}_{\bm{v}_{_0}}(M)} Y_{\bm{v}}\setminus \overline{\mathcal{D}},
		\end{equation}
		where the index set $\widetilde{\Pi}_{\bm{v}_{_0}}(M)\subset \Lambda$ is given by 
		\begin{equation}\label{eqn:square_index}
			\widetilde{\Pi}_{\bm{v}_{_0}}(M) := \{ m\mathring{\bm{v}}_{_1}+n\mathring{\bm{v}}_{_2}: |m-m_0|\le M,|n-n_0|\le M \},\quad \bm{v}_{_0} = m_0\mathring{\bm{v}}_{_1}+n_0\mathring{\bm{v}}_{_2}.
		\end{equation}
	\end{definition}
	For a given patch $\mathsf{P}$ and fixed $\bm{v}_{_0}\in \Pi$, we set 
	\[	M_{\mathsf{P}}(\bm{v}_{_0}):=
	\max\Big\{ M\in\mathbb{N}:	 \widetilde{\mathsf{P}}_{\bm{v}_{_0}}(M)\subset \mathsf{P}\Big\},\]
	and define the size of a patch as follows.
	\begin{definition}
		The \emph{size} of a patch $\mathsf{P}$ is defined by
		\begin{equation}		
			M_{\mathsf{P}}:= \max_{\bm{v}\in\Pi} M_{\mathsf{P}}(\bm{v}).	
		\end{equation}
		Moreover, we say that $\bm{v}_{\ast}\in\Pi$ is a \emph{center} of $\mathsf{P}$ if
		\[
		M_{\mathsf{P}}(\bm{v}_{\ast})=M_{\mathsf{P}}.
		\]
	\end{definition}
	Note that for a given patch $\mathsf{P}$ of size $M$, there exists at least one center. For clarity and simplicity, we assume that the patches in the following part are properly chosen so that they have a unique center at $\bm{v}_{_i}\in\Lambda$.
	
	Furthermore, we assume that the patch $\mathsf{P}$ can be covered by a regularized patch of size $2M$ centered at $\bm{v}_{\ast}$:
	\begin{equation}\label{eqn:assump_cover}
		\mathsf{P}\subset \widetilde{\mathsf{P}}_{\bm{v}_{\ast}}(2M).
	\end{equation} 
	
	An illustration of a patch $\mathsf{P}$ of size $1$ that satisfies the above assumptions is given in \Cref{fig:patch}.
	\begin{figure}[htbp] 
		\centering
		\includegraphics[width = 0.5\textwidth]{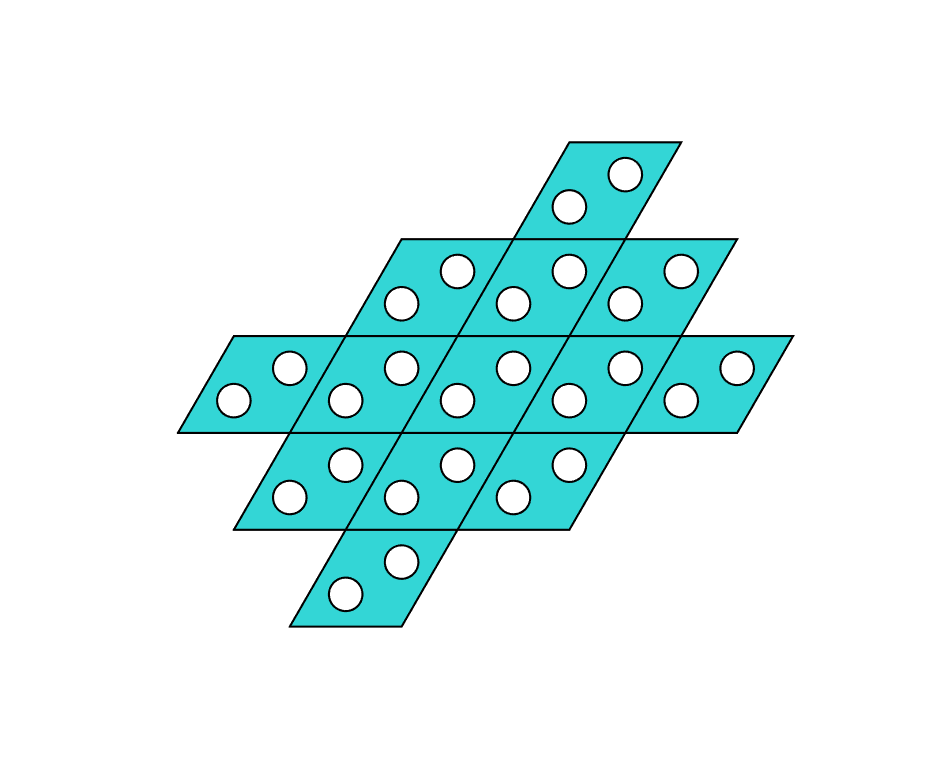}
		\caption{An illustration of a patch of size $1$.}
		\label{fig:patch}
	\end{figure}
	\begin{definition}
		For fixed $M\in\mathbb{N}$, we say that $\Omega$ is \emph{covered by finitely many patches of size $M$} if there exist $N\in\mathbb{N}$ and patches $\{ \mathsf{P}_{i} \}_{i=1}^{N}$ of size $M$, such that, for every $\bm{v}\in\Lambda$, there exist unique $1\le j\le N$ and $\bm{w}\in\Lambda$ with
		\[
		Y_{\bm{v}}\setminus \overline{\mathcal{D}} \subset \bm{w}+\mathsf{P}_j = \bigcup_{\bm{v}\in\Pi_{j}} Y_{\bm{v}}\setminus \overline{\mathcal{D}} \subset \Omega,
		\]
		and $\bm{v}$ is the center of the translated patch $\bm{w}+\mathsf{P}_j$. Here, for $1\leq j \leq N$, $\Pi_{j} \subset \Lambda$ is a finite index set.  
	\end{definition}
	In what follows, we assume that the domain $\Omega$ is covered by finitely many patches $\{ \mathsf{P}_{i} \}_{i=1}^{N}$ of size $M$. A domain $\Omega $ covered by finitely many patches of size $1$ is shown in \Cref{fig:finite_patch}.
	\begin{figure}[htbp] 
		\centering
		\includegraphics[width = 0.45\textwidth]{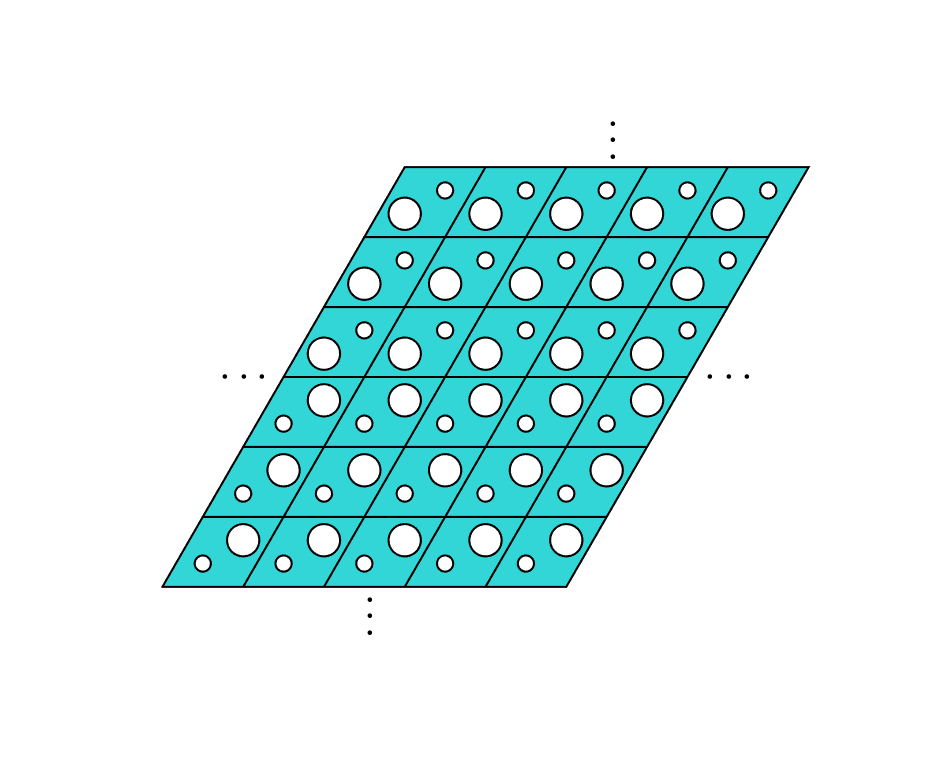}
		\hspace{-0.4cm}
		\includegraphics[width = 0.45\textwidth]{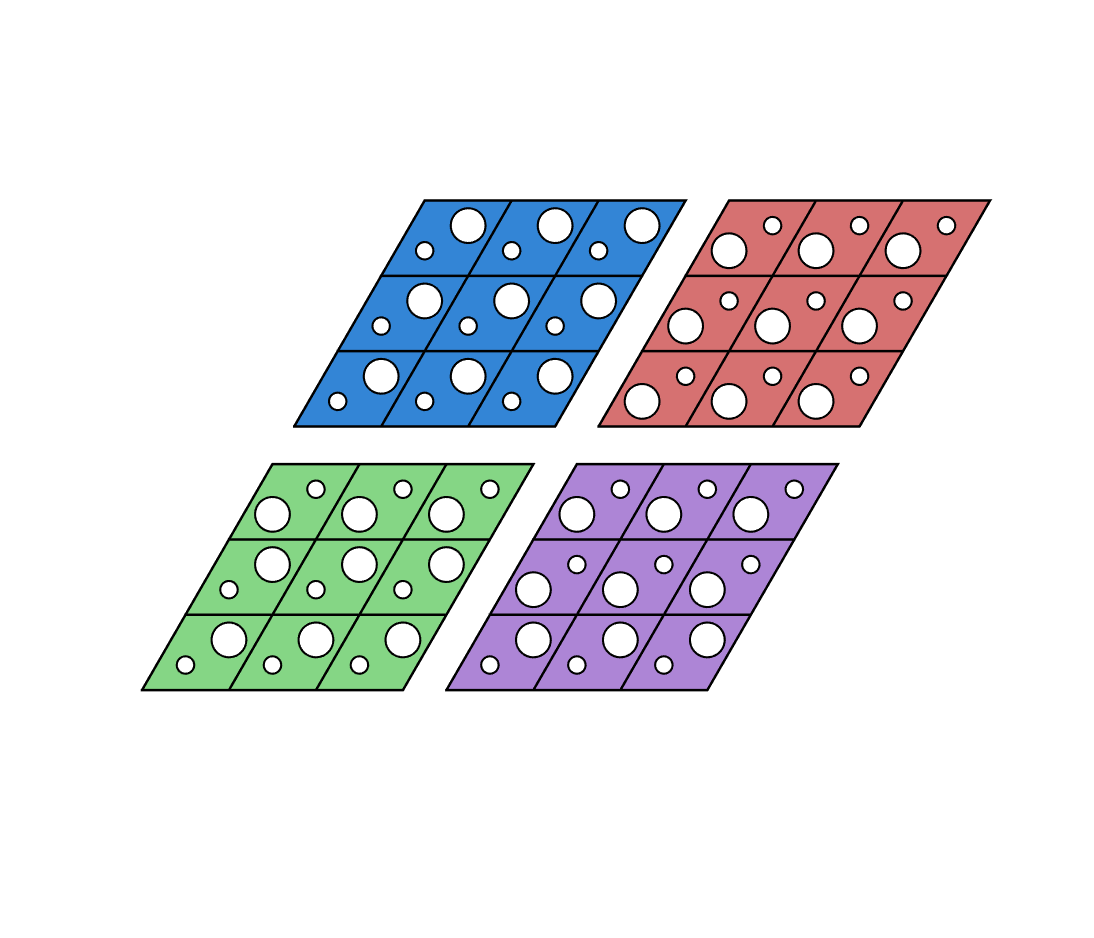}
		\caption{Left: An example of a finitely patched domain $\Omega$. Right: Four patches of size $1$ that correspond to the domain $\Omega$.}
		\label{fig:finite_patch}
	\end{figure}
	Unlike the periodic setting commonly studied via the Floquet--Bloch theory, we do not assume global periodicity of the domain $\Omega$.

	\subsection{Patch approximation of the capacitance operator}
	
	We now define the patch approximation of the capacitance operator. 
	From the definition of the Dirichlet-to-Neumann map $\mathcal{T}^0$, the capacitance operator introduced in \eqref{eq_cap_operator_def} can be equivalently defined by 
	\[ \mathcal{C}(\sigma,\sigma') = -\int_{\partial D_{\sigma'}} \frac{\partial V_{\sigma}}{\partial \nu}(y)\du \sigma(y), \quad \text{ for any } \sigma, \sigma' \in \Sigma, \]
	where $V_{\sigma}$ is the solution to the following boundary value problem:
	\begin{gather}\label{eqn:HarmonicFunc_full}
		\left\{\begin{aligned}
			\Delta V_{\sigma} &= 0 ,\quad x\in\Omega,\\
			V_{\sigma} & = 1,\quad x\in \partial D_{\sigma},\\
			V_{\sigma} & = 0,\quad x\in \partial \Omega \setminus \partial D_{\sigma}.
		\end{aligned}\right.
	\end{gather}
	
	Instead of solving the exterior problem on the whole domain $\Omega$, we solve local boundary value problems on the finitely many patches $\{ \mathsf{P}_{j} \}_{j=1}^N$ and then translate the solution to derive the corresponding local solutions for all $\sigma\in\Sigma$.
	\begin{definition}[Patch approximation of the capacitance operator]
		Assume that $\Omega$ is covered by finitely many patches $\{\mathsf{P}_j\}_{j=1}^{N}$ of size $M$. For each patch $\mathsf{P}_j$, let $\bm{v}_{_j}$ be its center. For $1\le j\le N$ and $1\le k\le K$, set $\sigma_j^k:=(\bm{v}_{_j},k)\in\Sigma$ and let $\mathcal{U}_{j,k}$ be the solution of
		\begin{gather}\label{eqn:Patch:primitive}
			\left\{\begin{aligned}
				\Delta\mathcal{U}_{j,k} &= 0,\quad x\in \mathsf{P}_j,\\
				\mathcal{U}_{j,k} &= 1,\quad x\in\partial D_{\sigma_j^k},\\
				\mathcal{U}_{j,k} &= 0,\quad x\in \partial \mathsf{P}_j\setminus \partial D_{\sigma_j^k}.
			\end{aligned}\right.
		\end{gather}
		For a general index $\sigma=(\bm{v},k)\in\Sigma$, let $\bm{w}+\mathsf{P}_j$ be the translated patch whose center is $\bm{v}$. We define
		\begin{equation}\label{eqn:PatchEqn}
			U_{\sigma}(x):=\mathcal{U}_{j,k}(x-\bm{w}),\quad x\in \bm{w}+\mathsf{P}_j .
		\end{equation}
		Equivalently, $U_{\sigma}$ solves
		\[
		\left\{\begin{aligned}
			\Delta U_{\sigma} &= 0,\quad x\in \bm{w}+\mathsf{P}_j,\\
			U_{\sigma} &= 1,\quad x\in \partial D_{\sigma},\\
			U_{\sigma} &= 0,\quad x\in \partial(\bm{w}+\mathsf{P}_j)\setminus \partial D_{\sigma}.
		\end{aligned}\right.
		\]
		For $\sigma'=(\bm{v}',k')\in\Sigma$, the \emph{patch approximation $\widetilde{\mathcal{C}}$} is defined by
		\begin{gather}\label{eqn:patchApprox}
			\widetilde{\mathcal{C}}(\sigma,\sigma') :=
			\left\{\begin{aligned}
				&-\int_{\partial D_{\sigma'} }\frac{\partial U_{\sigma}}{\partial \nu }(y) \, \du\sigma(y)
				&&\text{if }Y_{\bm{v}'}\setminus \overline{\mathcal{D}}\subset \bm{w}+\mathsf{P}_j, \\
				&0
				&&\text{otherwise}.
			\end{aligned} \right.
		\end{gather}
	\end{definition}
	
	The above definition shows why the patch approximation is computationally efficient. Only the $K\times N$ local problems in \eqref{eqn:Patch:primitive} need to be solved; all other local solutions are obtained from these by translation. The resulting operator keeps the interactions between a resonator and the resonators contained in its associated patch and sets all remaining interactions to zero.
	\begin{remark}
		We emphasize here that the patch approximation of the capacitance operator $\mathcal{C}$ is generally not symmetric. In addition, the definition of the patch approximation of the capacitance operator $\widetilde{\mathcal{C}}$ depends on the specific choice of the translated patch $\bm{w}+\mathsf{P}_j$ for all $\sigma\in\Sigma$. In later analysis, we specify the patch choice for each $\sigma\in\Sigma$ to ensure that the operator $\widetilde{\mathcal{C}}$ is determined uniquely.
	\end{remark}
	
	\section{Error analysis of the patch approximation}\label{sec:err-approx}
	In this section, we estimate the error of the patch approximation $\widetilde{\mathcal{C}}$ when the domain $\Omega$ is covered by finitely many patches $\{\mathsf{P}_i\}_{i=1}^N$ of size $M$. Specifically, for $\sigma=(\bm{v},k)\in\Sigma$ and $\sigma'=(\bm{v}',k') \in\Sigma$, we let $\bm{w} + \mathsf{P}_j$ be a patch of $\Omega$ centered at $\bm{v}$ and first consider the patch-wise error:
	\begin{equation}\label{eqn:patch_err_def}
		\mathcal{E}(M,\sigma):=  \sum_{\sigma':\bm{v}'\in \Pi_j}\int_{\partial (\bm{w}+D_{\sigma'})} \Big[\frac{\partial V_{\sigma}}{\partial \nu}(y)-\frac{\partial U_{\sigma}}{\partial \nu}(y)\Big]\, \du\sigma(y).
	\end{equation}
	By the maximum principle of $V_{\sigma}-U_{\sigma}$ on $\bm{w} + \mathsf{P}_j$, we have
	\begin{align*}
		V_{\sigma}(x) \ge \tau_{\bm{w}}\mathcal{U}_{j,l}(x) = U_{\sigma}, 
		\quad x\in\bm{w}+\mathsf{P}_j.
	\end{align*}
	Therefore, the error is no less than zero, and we have the following estimate for $\mathcal{E}(M,\sigma)$.
	\begin{lemma}\label{lem:patch-entry-error}
		If the domain $\Omega$ is covered by finitely many patches of size $M$, then there exist positive constants $C$ and $\rho\in(0,1)$ such that
		\begin{equation}\label{eqn:patch_error}
			\mathcal{E}(M,\sigma) \le C\rho^{M},\quad \forall \sigma\in \Sigma.
		\end{equation}
	\end{lemma}
	
	\begin{proof}
		Since $V_{\sigma}-U_{\sigma}$ is harmonic in $\bm{w}+\mathsf{P}_j$, we have
		\begin{equation}\label{eqn:normal_exterior}
			\mathcal{E}(M,\sigma) = \int_{\partial \Gamma_{\bm{w},j}} \Big[\frac{\partial V_{\sigma}}{\partial \nu}(y)-\frac{\partial U_{\sigma}}{\partial \nu}(y)\Big]\, \du\sigma(y),
		\end{equation}
		where 
		\[ \partial \Gamma_{\bm{w},j}:= \partial (\bm{w} + \mathsf{P}_j)\setminus\bigcup_{\bm{v}'\in \Pi_{j}}\bigcup_{k'=1}^{K}\partial (\bm{w}+D_{\sigma'})  \]
		denotes the exterior boundary of $\bm{w} + \mathsf{P}_j$. The normal vector $\nu $ on $ \partial \Gamma_{\bm{w},j}$ points outward from $\bm{w} + \mathsf{P}_j$. 
		By \Cref{aplem:normal_deri}, the right-hand side of \eqref{eqn:normal_exterior} is bounded by 
		\begin{equation}
			\mathcal{E}(M,\sigma) \le C|\partial \Gamma_{\bm{w},j} |  \sup_{\partial \Gamma_{\bm{w},j}}|V_{\sigma}-U_{\sigma}| = C|\partial \Gamma_{\bm{w},j} |  \sup_{\Gamma_{\bm{w},j}}V_{\sigma}.
		\end{equation} 
		The last equality holds since $U_{\sigma}$ vanishes on $ \Gamma_{\bm{w},j}$. 
		The estimate \eqref{apeqn:Greenfunc_estimate} for Green's function on $\Omega$, together with \eqref{apeqn:VGreen_rep}, implies that there exist positive constants $C$ and $c$ such that
		\begin{equation}
			\sup_{\Gamma_{\bm{w},j}}V_{\sigma} \le C \frac{\eu^{-c\operatorname{dist}(D_{\sigma},\partial \Gamma_{\bm{w},j})}}{\sqrt{\operatorname{dist}(D_{\sigma},\partial \Gamma_{\bm{w},j})}}.
		\end{equation} 
		Therefore, by \eqref{eqn:assump_cover}, we have 
		\[ \mathcal{E}(M,\sigma) \le 4C(4M+1)^2  \frac{\eu^{-c\operatorname{dist}(D_{\sigma},\partial \Gamma_{\bm{w},j})}}{\sqrt{\operatorname{dist}(D_{\sigma},\partial \Gamma_{\bm{w},j})}}.\]
		By the definition of a patch of size $M$ and the uniform separation constants in \Cref{asmp_geometry_assumption}, the resonator $D_\sigma$ is separated from the exterior boundary of the translated patch by a distance bounded below as
		\[
		\operatorname{dist}(D_\sigma,\partial\Gamma_{\bm w,j})\ge c_0M-c_1
		\]
		for constants $c_0,c_1>0$ independent of $M$ and $\sigma$. Hence, the polynomial factor $(4M+1)^2$ can be absorbed into the exponential. Combining these estimates, there exist positive constants $C$ and $\rho\in(0,1)$ such that
		\[ \mathcal{E}(M,\sigma)\le 4C(4M+1)^2  \frac{\eu^{-c\operatorname{dist}(D_{\sigma},\partial \Gamma_{\bm{w},j})}}{\sqrt{\operatorname{dist}(D_{\sigma},\partial \Gamma_{\bm{w},j})}}\le C \rho^{M}, \]
		which concludes the proof of \eqref{eqn:patch_error}.
		
	\end{proof}
	
	Next, we prove uniform estimates for the entries of $\mathcal{C}-\widetilde{\mathcal{C}}$.
	
	\begin{lemma}\label{lem:block_sum_estimate}
		If the domain $\Omega$ is covered by finitely many patches of size $M$, then there exist constants $C>0$ and $\widetilde{\rho}\in(0,1)$, independent of $M$, such that
		\begin{equation}\label{eqn:col_sum_scalar}
			\sup_{\sigma\in\Sigma}
			\sum_{\sigma'\in\Sigma}
			|(\mathcal{C}-\widetilde{\mathcal{C}})(\sigma,\sigma')|
			\le C\widetilde{\rho}^M,
		\end{equation}
		and
		\begin{equation}\label{eqn:row_sum_scalar}
			\sup_{\sigma'\in\Sigma}
			\sum_{\sigma\in\Sigma}
			|(\mathcal{C}-\widetilde{\mathcal{C}})(\sigma,\sigma')|
			\le C\widetilde{\rho}^M.
		\end{equation}
	\end{lemma}
	
	\begin{proof}
		We first prove \eqref{eqn:col_sum_scalar}. For given $\sigma = (\bm{v},k)\in \Sigma$ and $\sigma' = (\bm{v}',k')\in \Sigma$, let $\bm{w}+\mathsf{P}_j$ be the translated patch centered at the cell $\bm{v}$. By definition, we have
		\begin{align*}
			\sum_{\sigma'\in\Sigma}
			|(\mathcal{C}-\widetilde{\mathcal{C}})(\sigma,\sigma')|
			&= \sum_{\sigma':\bm{v}'\in \Pi_{j}}
			|(\mathcal{C}-\widetilde{\mathcal{C}})(\sigma,\sigma')|+\sum_{\sigma':\bm{v}'\in \Lambda\setminus\Pi_{j}}
			|\mathcal{C}(\sigma,\sigma')| \\
			&= \mathcal{E}(M,\sigma)+\sum_{\sigma':\bm{v}'\in\Lambda\setminus\Pi_{j}}\int_{\partial (\bm{w}+D_{\sigma'})} \frac{\partial V_{\sigma}}{\partial \nu}(y)\, \du \sigma(y).
		\end{align*}
		The first term is bounded by $C\rho^M$ by
		\Cref{lem:patch-entry-error}. For the
		second term, we obtain using \Cref{apthm:exp_decay}
		\[
		\sum_{\sigma':\bm{v}'\in\Lambda\setminus\Pi_{j}}|\mathcal{C}(\sigma,\sigma')|
		\le C\rho^M.
		\]
		This proves \eqref{eqn:col_sum_scalar}.
		
		We next prove \eqref{eqn:row_sum_scalar}. For fixed $\sigma'=(\bm{v}',k')\in\Sigma$, consider $\sigma=(\bm{v},k)\in\Sigma$ and choose a translated patch $\bm{w}(\sigma) + \mathsf{P}_{j(\sigma)}$ centered at $\bm{v}$. Define 
		\[
		\Xi_{\sigma'}:=	\{ \sigma\in \Sigma: Y_{\bm{v}'}\setminus \overline{\mathcal{D}} \subset \bm{w}(\sigma) + \mathsf{P}_{j(\sigma)} \}.
		\]
		Then
		\[
		\sum_{\sigma\in\Sigma}
		|(\mathcal{C}-\widetilde{\mathcal{C}})(\sigma,\sigma')|
		\le \sum_{\sigma\in\Xi_{\sigma'}}
		|(\mathcal{C}-\widetilde{\mathcal{C}})(\sigma,\sigma')| + \sum_{\sigma\in\Sigma\setminus\Xi_{\sigma'}}
		|\mathcal{C}(\sigma,\sigma')|.
		\]
		The second term is also bounded by $C\rho^M$ by
		\Cref{apthm:exp_decay}. For the first term, the covering assumption \eqref{eqn:assump_cover} implies that $\# \, \Xi_{\sigma'}\le (4M+1)^2$. Moreover, \Cref{lem:patch-entry-error} gives
		\[ |(\mathcal{C}-\widetilde{\mathcal{C}})(\sigma,\sigma')|\le C\rho^{M}, \]
		for $\sigma\in \Xi_{\sigma'}$. Therefore,
		\[
		\sum_{\sigma\in\Xi_{\sigma'}}|(\mathcal{C}-\widetilde{\mathcal{C}})(\sigma,\sigma')|
		\le C(4M+1)^2\rho^M
		\le C\widetilde\rho^M,
		\]
		after replacing $\rho$ with a larger number
		$\widetilde\rho\in(\rho,1)$. This proves
		\eqref{eqn:row_sum_scalar}.
	\end{proof}
	
	\subsection{Norm error bound for the patch approximation}
	Now, we derive the operator-norm error bound for the patch approximation of the capacitance operator.
	\begin{theorem}\label{thm:op_error_capacitance}
		If the domain $\Omega$ is covered by finitely many patches of size $M$, then there exist constants $C>0$ and $\widetilde{\rho}\in(0,1)$, which are independent of $M$, such that
		\begin{equation}
			\|\mathcal{C}-\widetilde{\mathcal{C}}\|_{\mathcal{B}(\ell^2(\Sigma))}
			\le C\widetilde{\rho}^{M}.
		\end{equation}
	\end{theorem}
	
	\begin{proof}
		For any $\boldsymbol{\Phi} = \{\Phi(\sigma)\}_{\sigma\in\Sigma}\in \ell^2(\Sigma)$ and fixed $\sigma'\in\Sigma$,
		\[
		|((\mathcal{C}-\widetilde{\mathcal{C}})\boldsymbol{\Phi})({\sigma'})|
		\le
		\sum_{\sigma\in\Sigma}
		|(\mathcal{C}-\widetilde{\mathcal{C}})(\sigma,\sigma')|\,|\Phi(\sigma)|.\]
		By the weighted Cauchy-Schwarz inequality, it follows that
		\begin{gather*}
			|((\mathcal{C}-\widetilde{\mathcal{C}})\boldsymbol{\Phi})(\sigma')|^2
			\le
			\Big(\sum_{\sigma\in \Sigma}|(\mathcal{C}-\widetilde{\mathcal{C}})(\sigma,\sigma')|\Big)\Big(\sum_{\sigma\in\Sigma}|(\mathcal{C}-\widetilde{\mathcal{C}})(\sigma,\sigma')|\,|\Phi(\sigma)|^2\Big).
		\end{gather*}
		Summing over $\sigma'\in\Sigma$, and using \eqref{eqn:col_sum_scalar} and \eqref{eqn:row_sum_scalar}, we obtain
		\begin{align*}
			\|(\mathcal{C}-\widetilde{\mathcal{C}})\boldsymbol{\Phi}\|_{\ell^2(\Sigma)}^2
			&\le
			C\widetilde{\rho}^M\sum_{\sigma'\in\Sigma}\sum_{\sigma\in\Sigma}
			|(\mathcal{C}-\widetilde{\mathcal{C}})(\sigma,\sigma')|\,|\Phi(\sigma)|^2\\
			&=
			C\widetilde{\rho}^M\sum_{\sigma\in\Sigma}\|\Phi(\sigma)\|^2
			\sum_{\sigma'\in\Sigma}|(\mathcal{C}-\widetilde{\mathcal{C}})(\sigma,\sigma')|\\
			&\le C^2\widetilde{\rho}^{2M}\|\boldsymbol{\Phi}\|_{\ell^2(\Sigma)}^2.
		\end{align*}
		This gives
		\[
		\|(\mathcal{C}-\widetilde{\mathcal{C}})\boldsymbol{\Phi}\|_{\ell^2(\Sigma)}
		\le C\widetilde{\rho}^M\|\boldsymbol{\Phi}\|_{\ell^2(\Sigma)}.
		\]
		Hence,
		\[
		\|\mathcal{C}-\widetilde{\mathcal{C}}\|_{\mathcal{B}(\ell^2(\Sigma))}
		\le C\widetilde{\rho}^M,
		\]
		which completes the proof.
	\end{proof}

	\begin{remark}
		In the general case $|D_{\sigma}| \neq 1$, we solve the following generalized eigenvalue problem:
		\begin{equation}\label{eqn:gen_eigvalue}
			\mathcal{C}\boldsymbol{\Phi} = z\mathcal{M}\boldsymbol{\Phi}.
		\end{equation} 
		Here, the operator $\mathcal{M}:\ell^2(\Sigma)\to \ell^2(\Sigma)$ is given by 
		\begin{equation}\label{eqn:MassOperator}
			(\mathcal{M}\boldsymbol{\Phi})(\sigma) = |D_{\sigma}|\Phi({\sigma}),
		\end{equation}
		where $|D_{\sigma}|$ denotes the volume of $D_{\sigma}.$
		The proof of \Cref{thm_cont_to_disc_resolvent_converge} also holds by normalizing the projection operator on each resonator as follows: 
		\[ (Pu)(\sigma) = \frac{1}{|D_{\sigma}|}\int_{D_{\sigma}} u\, \du x. \]
		
	\end{remark}
	
	\section{Numerical experiments}\label{sec:num_patch}
	In this section, we present numerical results that validate our theoretical analysis.
	To assemble the capacitance operator, we discretize \eqref{eqn:Patch:primitive} with a Galerkin boundary element method based on a single-layer potential representation of the solution. We then use Colbrook's rectangular truncation method \cite{Colbrook2019,Colbrook2023} to compute localized eigenmodes while avoiding spectral pollution. 
	Details are given in \Cref{apsec:numericalmethod}. For eigenmodes that lie in absolutely continuous spectra, we will study them in future work.
	
	For simplicity, we first consider the square lattice with
	\begin{equation}\label{eqn:square_lattice}
		\mathring{\bm{v}}_{_1} = (1,0)^\top,\quad \mathring{\bm{v}}_{_2} = (0,1)^\top,
	\end{equation}
	and assume that there is only one polygon $Y^{(1)}$ given by 
	\[ Y^{(1)}= \{ s\mathring{\bm{v}}_{_1}+t\mathring{\bm{v}}_{_2}:s,t\in[-1/2,1/2)\}. \]
	Each translated polygon $Y_{\bm{v}}$ contains a single disk $D_{\sigma} = B(\bm{v},r_{\bm{v}})$ centered at $\bm{v}$ with radius $r_{\bm{v}}$, where $\sigma=(\bm{v},1)$.
	
	We also consider generalized honeycomb structures \cite{Miao2024,Cao2025} to validate our method in complex structures; see \Cref{exmp_resolvent_converge_interface}. In this setting, we recall that  
	\begin{equation}\label{eqn:honeycomb}
		\mathring{\bm{v}}_{_1} = (\sqrt{3}/2,-1/2)^\top,\quad \mathring{\bm{v}}_{_2} = (\sqrt{3}/2,1/2)^\top.
	\end{equation}
	
	\subsection{Numerical validation of the patch approximation}
	We first validate the patch approximation $\widetilde{\mathcal{C}}$ by computing its coefficients and comparing them with a reference solution obtained from a large patch size $M$.
	We consider
	\[ \widetilde{\mathsf{P}}_{\bm{0}}(M)= \bigcup_{\bm{v}\in\widetilde{\Pi}_{\bm{0}}(M)} Y_{\bm{v}}\setminus \overline{\mathcal{D}}, \]
	where the index set $\widetilde{\Pi}_{\bm{0}}$ is given by \eqref{eqn:square_index}. 
	The radius $r_{\bm{v}}$ is set to $r_{\bm{v}} = 0.3$ for all $\bm{v}\in\Lambda$. We solve \eqref{eqn:Patch:primitive} on $\widetilde{\mathsf{P}}_{\bm{0}}(M)$ and take the numerical solution at $M=10$ as the reference solution. We calculate the patch-wise error $\mathcal{E}(M,\sigma) = \mathcal{E}(M)$, and the results are shown in \Cref{fig:err_patch_element}.
	\begin{figure}[htbp] 
		\centering
		\includegraphics[width = 0.45\textwidth]{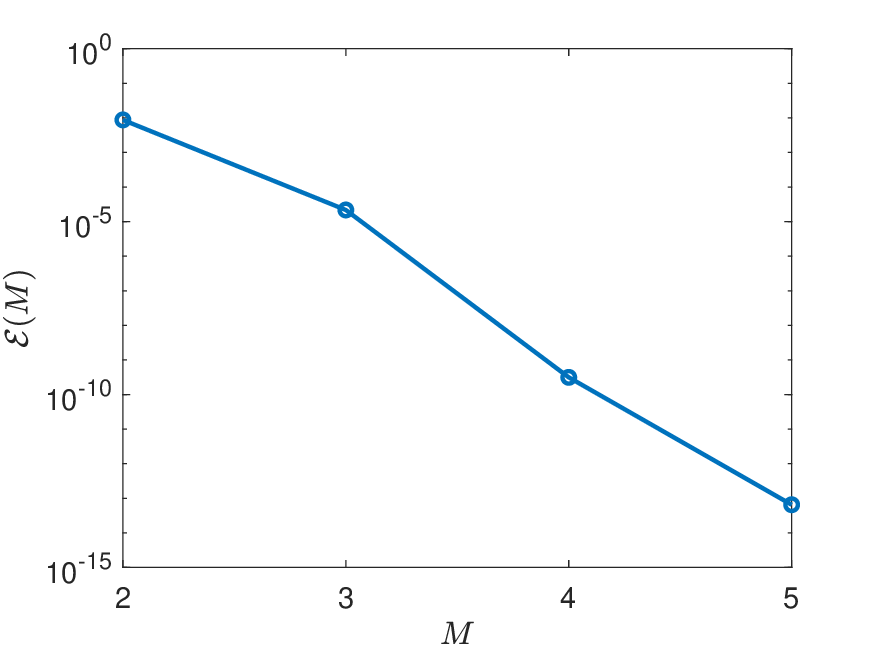}
		\caption{An illustration of element-wise error.}
		\label{fig:err_patch_element}
	\end{figure}
	
	As shown in \Cref{fig:err_patch_element}, the patch approximation converges exponentially to the reference solution by element.

	\subsection{Acceleration of the patch approximation}
	In this subsection, we show that the patch approximation substantially accelerates the assembly of the capacitance operator. We consider disk radii $r_{\bm{v}}$ satisfying
	\begin{gather*}
		r_{\bm{v}} = \left\{\begin{aligned}
			0.15,\quad \bm{v}=(0,0)^{\top},\\
			0.3,\quad \bm{v}\neq (0,0)^{\top}.
		\end{aligned}\right. 
	\end{gather*}
	This configuration supports a defect mode localized near $\bm{0}$. We compute the patch approximation of the capacitance operator $\widetilde{\mathcal{C}}$ and adopt, as described before, the rectangular truncation method to calculate the eigenvalue of the defect mode. 
	
	Specifically, let 
	\[\Sigma_R:=\{(m\mathring{\bm{v}}_{_1}+n\mathring{\bm{v}}_{_2},1): \max(|m|,|n|)\le R\}\]
	denote the truncated index set for the capacitance operator.
	We first calculate the patch approximation $\widetilde{\mathcal{C}}$ of the capacitance operator $\mathcal{C}$ by choosing the regularized patch of size $M = 2,3,4,5$. Then, we truncate $\widetilde{\mathcal{C}}$ by computing
	\[ 
	P_{\mathcal{N}'} (\widetilde{\mathcal{C}}-z\mathcal{M}) P_{\mathcal{N}},
	\quad \mathcal{N} = \Sigma_{8},\quad \mathcal{N}' = \Sigma_{8+M}.
	\]
	Here, the operator $\mathcal{M}$ is given by \eqref{eqn:MassOperator} and 
	$P_{\mathcal{N}}$ and $P_{\mathcal{N}'}$ are the projections on the computational domains $\mathcal{N}$ and $\mathcal{N}'$, respectively; see \Cref{subsec:pn}.
	
	The reference solution is obtained by solving the eigenvalue problem for the following truncated matrix: 
	\[ P_{\mathcal{N}'_{\mathrm{ref}}} (\mathcal{C}_{\mathrm{ref}}-z\mathcal{M}) P_{\mathcal{N}_{\mathrm{ref}}},
	\quad \mathcal{N}_{\mathrm{ref}} = \Sigma_{8},\quad \mathcal{N}'_{\mathrm{ref}} = \Sigma_{16}, \]
	Each element of $P_{\mathcal{N}'_{\mathrm{ref}}} (\mathcal{C}_{\mathrm{ref}}-z\mathcal{M}) P_{\mathcal{N}_{\mathrm{ref}}}$ is given by
	\begin{gather*}
		\mathcal{C}_{\mathrm{ref}}(\sigma,\sigma') = -\int_{\partial D_{\sigma'}}\frac{\partial U_{\sigma,\mathrm{ref}}}{\partial \nu}(y) \, \du \sigma(y),\quad \sigma\in\mathcal{N}_{\mathrm{ref}},\sigma'\in\mathcal{N}'_{\mathrm{ref}}.
	\end{gather*}
	Here, $U_{\sigma,\mathrm{ref}}$ is the solution to the following boundary value problem:
	\[
	\left\{\begin{aligned}
		\Delta U_{\sigma,\mathrm{ref}} &= 0,\quad x\in \widetilde{\mathsf{P}}_{\bm{0}}(16),\\
		U_{\sigma,\mathrm{ref}} &= 1,\quad x\in \partial D_{\sigma},\\
		U_{\sigma,\mathrm{ref}} &= 0,\quad x\in \partial\widetilde{\mathsf{P}}_{\bm{0}}(16)\setminus \partial D_{\sigma}.
	\end{aligned}\right.
	\]
	
	\Cref{table:compareDefect} compares the eigenvalue error, runtime, and memory consumption for the construction of the patch approximation. The results show that $M=2$ already provides sufficient accuracy for this example. 
	\begin{table}[htbp]
		\centering
		\caption{Comparison of eigenvalue error, runtime, and memory consumption when building the patch approximation.}
		\begin{tabular}{cccc}
			\toprule
			& {$|z_{M}-z_{\mathrm{ref}}|$}  &  Runtime   &  Memory \\ 
			\midrule
			$M=2$ & $2.391610\times 10^{-5}$ & 1.562 s   & 34.10 MB\\
			$M=3$ & $1.749009\times 10^{-8}$ & 10.403 s  & 97.09 MB\\
			$M=4$ & $1.234923\times 10^{-10}$ & 48.206 s  & 271.29 MB\\
			$M=5$ & $1.350031\times 10^{-13}$ & 160.768 s & 530.58 MB\\
			Reference & -- & 756.965 s & 26892.22 MB\\
			\bottomrule
		\end{tabular}
		\label{table:compareDefect}
	\end{table}

	\subsection{Square defect} 
	In this subsection, we consider the square defect given by
	\begin{gather*}
		r_{\bm{v}} = \left\{\begin{aligned}
			&0.15,\quad -8\le m\le 8,|n|=8, \text{ or} -8\le n\le 8,|m|=8, \\
			&0.3,\quad (m,n) \text{ otherwise},
		\end{aligned}\right. 
	\end{gather*}
	where $\bm{v} = m\mathring{\bm{v}}_{_1}+n\mathring{\bm{v}}_{_2}$.
	We build the patch approximation of size $M=2$ and choose
	\[  \mathcal{N}=\Sigma_{16} ,\quad \mathcal{N}'=\Sigma_{18} . \]
	The corresponding rectangular matrix is $P_{\mathcal{N}'} (\widetilde{\mathcal{C}}-z\mathcal{M}) P_{\mathcal{N}}$.
	We plot some representative localized eigenmodes near the defect. \Cref{fig:square_wave} shows the absolute value of the corresponding eigenvectors. The value of $F_{\mathcal{N}}(z)$ defined in \eqref{def:fz} is also provided to show that this method gives an accurate approximation of the localized eigenmodes.
	\begin{figure}[h!] 
		\centering
		\includegraphics[width = 0.45\textwidth]{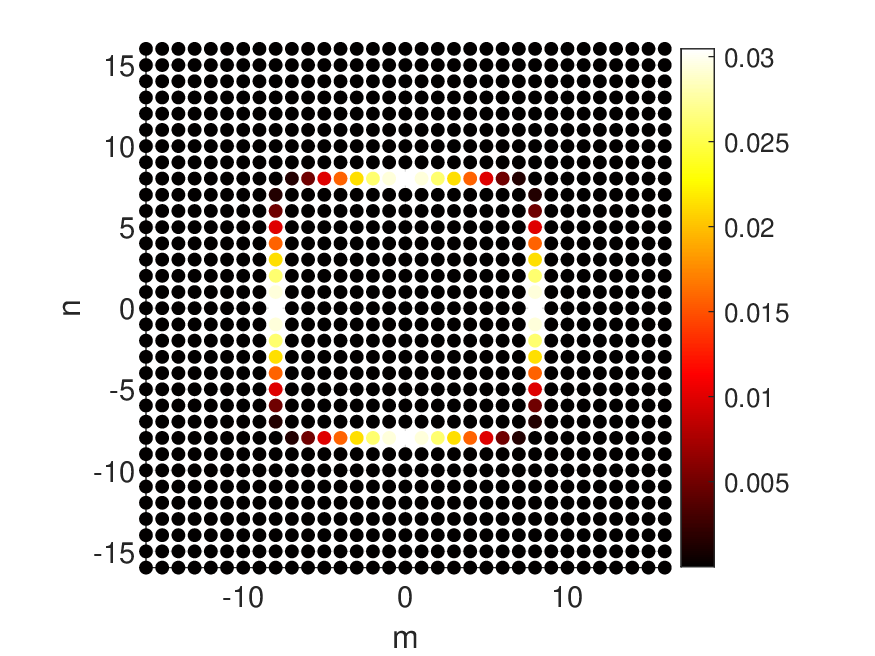}
		\hspace{-0.2cm}
		\includegraphics[width = 0.45\textwidth]{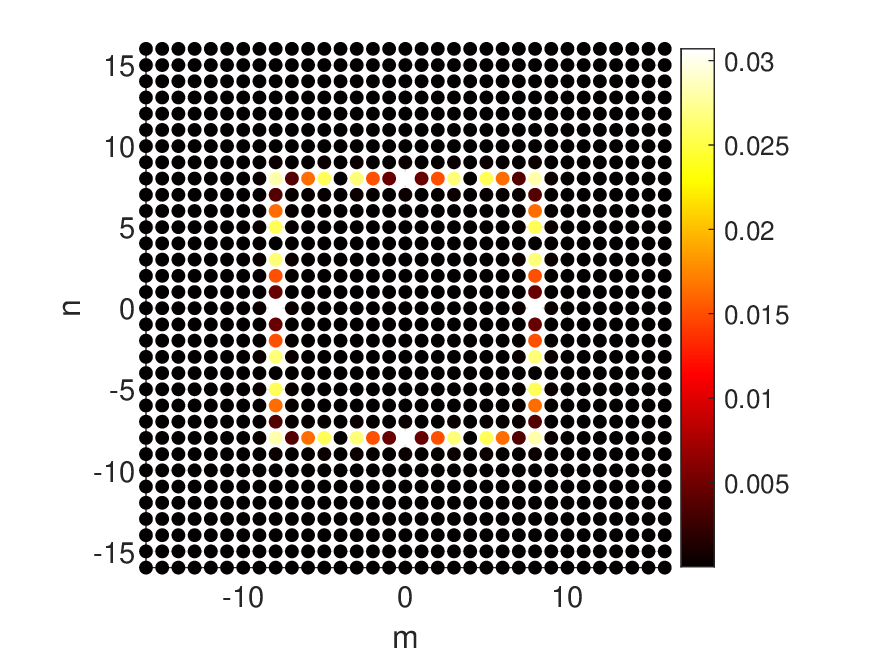}
		\caption{Localized eigenmodes near the square defect. Left panel: $z= 63.565, F_{\mathcal{N}}(z) = 1.603\times10^{-8}$. Right panel: $ z= 47.561, F_{\mathcal{N}}(z) = 1.282\times 10^{-6}$.}  
		\label{fig:square_wave}
	\end{figure}

	\subsection{Bent waveguides}
	In this subsection, we consider the case of a bent waveguide given by
	\begin{gather*}
		r_{\bm{v}} = \left\{\begin{aligned}
			&0.15,\quad  m\le 0,n=0, \text{ or }  n\le 0,m=0, \\
			&0.3,\quad (m,n) \text{ otherwise},
		\end{aligned}\right. 
	\end{gather*}
	where $\bm{v} = m\mathring{\bm{v}}_{_1}+n\mathring{\bm{v}}_{_2}$.
	We build the patch approximation of size $M=2$ and choose
	\[  \mathcal{N}=\Sigma_{16} ,\quad \mathcal{N}'=\Sigma_{18} . \]
	And the corresponding rectangular matrix is $P_{\mathcal{N}'} (\widetilde{\mathcal{C}}-z\mathcal{M}) P_{\mathcal{N}}$. \Cref{fig:bending_wave} shows both guided modes inside the bent waveguide and a corner mode. A detailed description of this corner mode is left for future work.
	
	\begin{figure}[h!] 
		\centering
		\includegraphics[width = 0.45\textwidth]{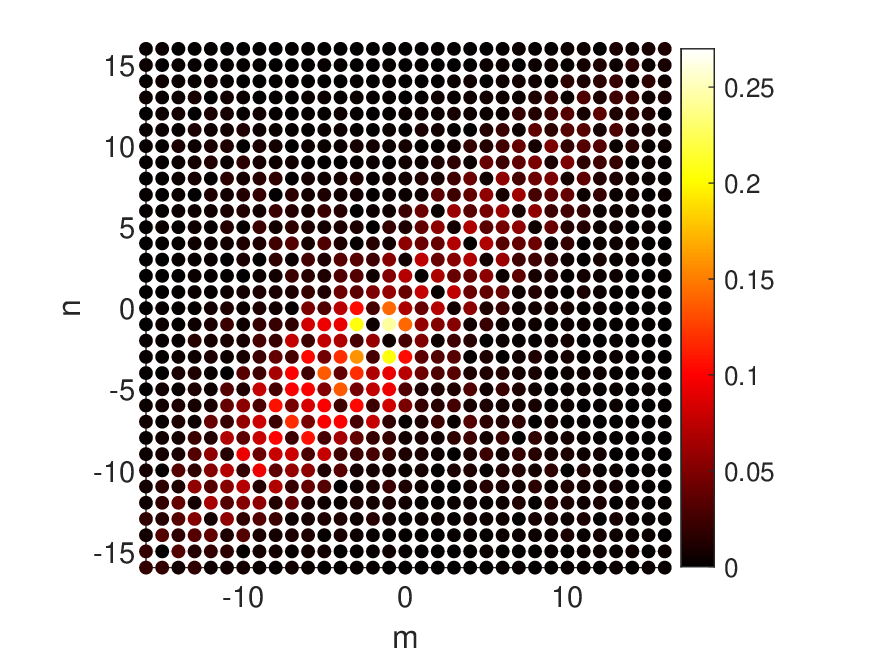}
		\hspace{-0.2cm}
		\includegraphics[width = 0.45\textwidth]{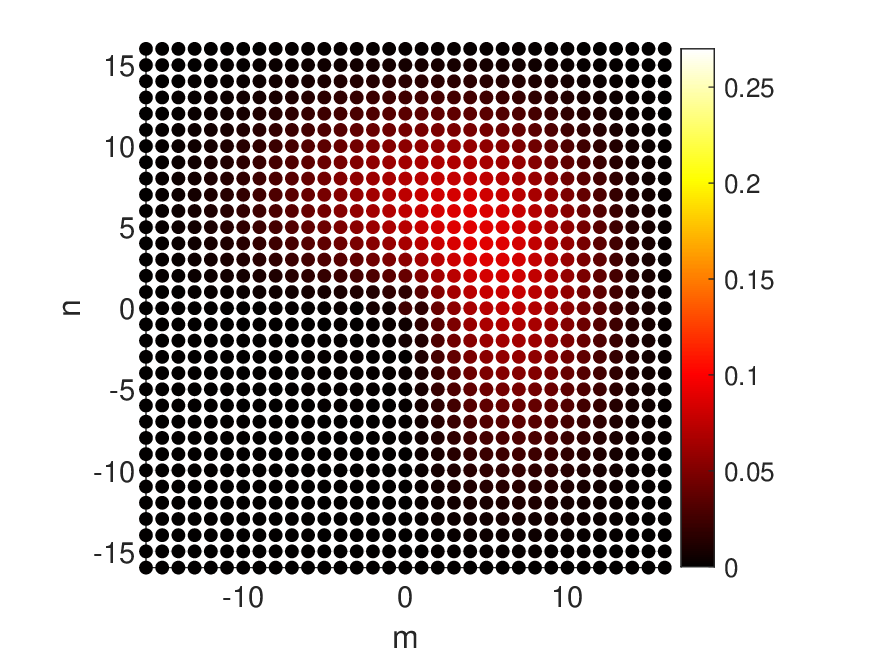}
		\includegraphics[width = 0.45\textwidth]{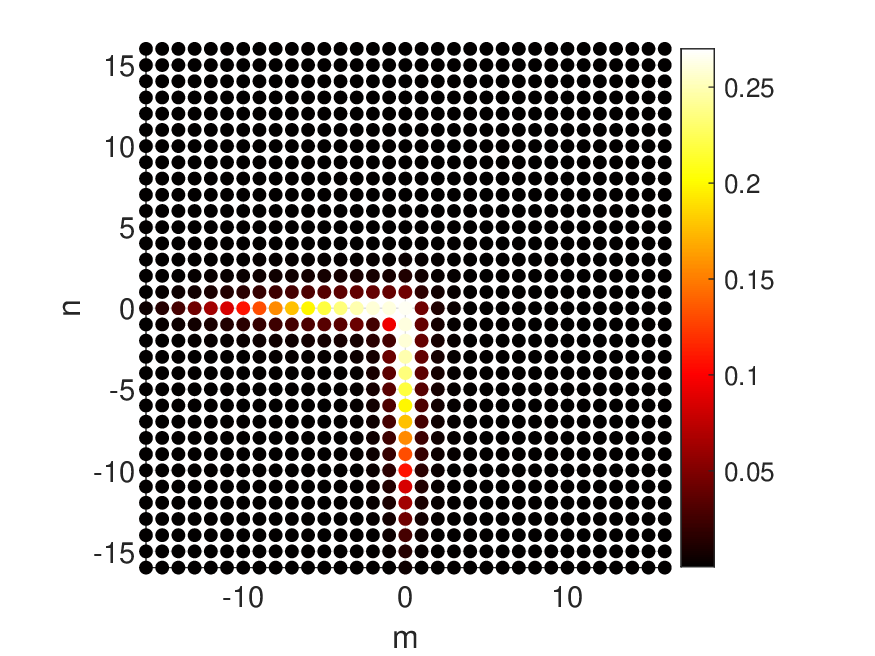}
		\hspace{-0.2cm}
		\includegraphics[width = 0.45\textwidth]{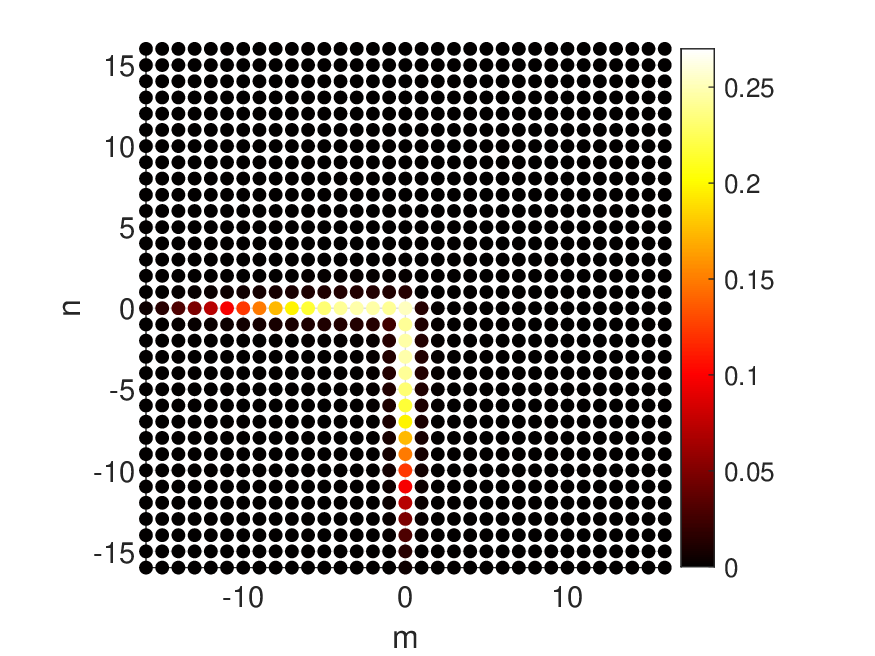}
		\caption{Localized eigenmodes in a bent waveguide. Top left panel: $z= 16.297, F_{\mathcal{N}}(z) = 2.706\times10^{-1}$. Top right panel: $ z= 37.658, F_{\mathcal{N}}(z) = 2.889\times 10^{-2}$. Bottom left panel: $ z= 40.903, F_{\mathcal{N}}(z) = 6.589\times 10^{-3}$. Bottom right panel: $ z= 63.752, F_{\mathcal{N}}(z) = 7.885\times 10^{-3}$.  }
		\label{fig:bending_wave}
	\end{figure}

	\subsection{Waveguide based on generalized honeycomb structures}
	In this subsection, we investigate waveguide structures based on generalized honeycomb lattices described in \Cref{exmp_interfaceMode}. 
	Such structures support a double Dirac point at the center of their corresponding Brillouin zone \cite{Miao2024,Cao2025}. 
	We consider type-I and type-II interfaces \cite{borui-cmp}, as illustrated in \Cref{fig:gen_honey}. Since the system is periodic in the $\bm{v}_{\alpha}$ direction, its spectrum can be computed using the Floquet-Bloch theory. However, this requires discretizing the operator for all quasi-periodicities $\alpha\in[-\pi,\pi)$, which leads to a relatively high computational cost. We compare the spectrum obtained by the Floquet-Bloch theory with that given by the patch approximation. As a reference, we use the supercell method for
	\[ \alpha = -\pi,-\pi+\pi/20,\ldots, \pi, \]
	and take the union over all $\alpha$ to approximate the full spectrum. We build the patch approximation of size $M=2$ and choose
	\[  
	\left\{\begin{aligned}
		\mathcal{N} &= \{ (m\bm{v}_{\alpha}+n\bm{v}_{\beta},k):\max(|m|,|n|)\le 30,\, 1\le k\le 6 \},\\
		\mathcal{N}' &= \{ (m\bm{v}_{\alpha}+n\bm{v}_{\beta},k):\max(|m|,|n|)\le 32,\, 1\le k\le 6 \}.
	\end{aligned}\right.
	\]
	And the corresponding rectangular matrix is $P_{\mathcal{N}'}(\widetilde{\mathcal{C}}-z\mathcal{M})P_{\mathcal{N}}$. \Cref{fig:super_wave} shows the edge modes in these structures.
	\begin{figure}[h!] 
		\centering
		\includegraphics[width = 0.2\textwidth]{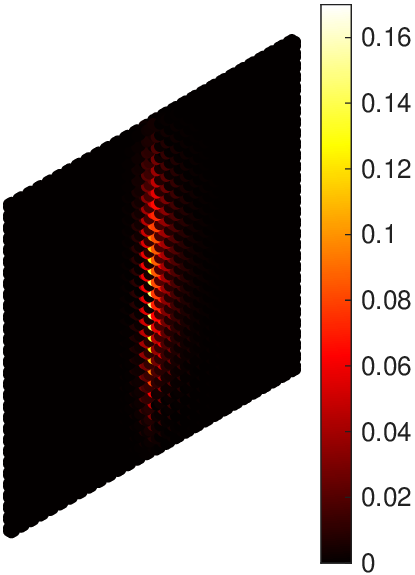}
		\hspace{0.5cm}
		\includegraphics[width = 0.48\textwidth]{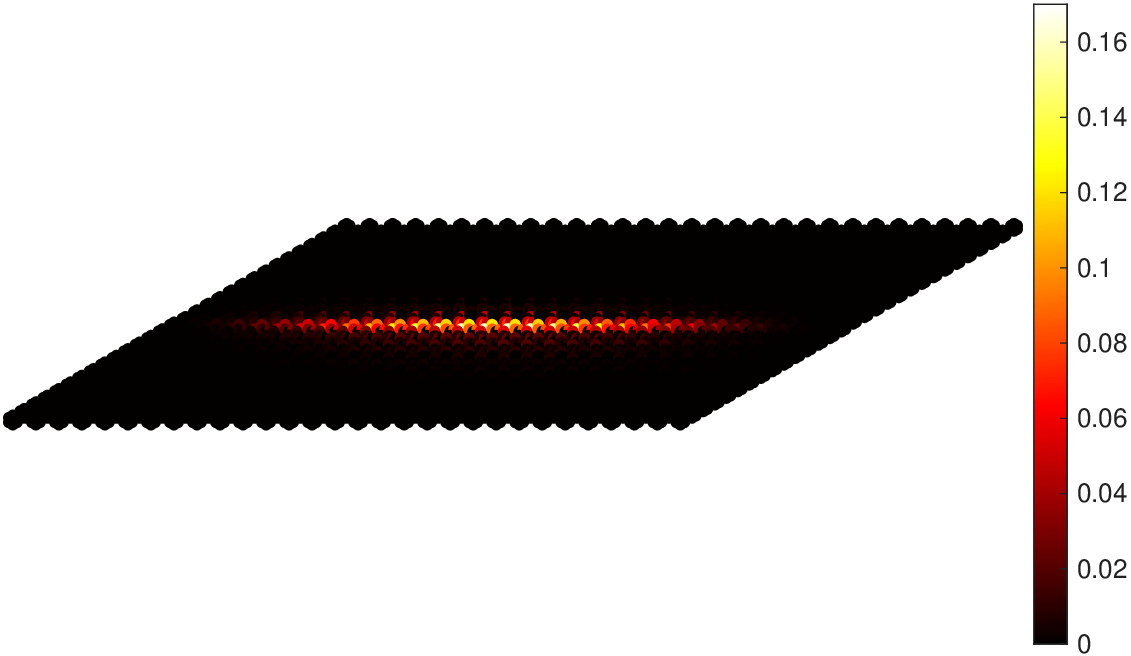}%
		\caption{Interface modes of type-I and type-II in a generalized honeycomb structure. Left panel: an example of interface modes of type-I. Right panel: an example of interface modes of type-II. }
		\label{fig:super_wave}
	\end{figure}
	
	The function $F_{\mathcal{N}}(z)$ defined by \eqref{def:fz} is plotted in \Cref{fig:compare_spectrum}. For comparison with the supercell method, we also plot the quasi-periodic spectrum for each $\alpha$ on the $z$-axis for $z\in\mathbb{R}$, with pseudo-modes being removed. The results show that the patch approximation accurately captures the band gap near the edge eigenmodes, without any spectral pollution.

	\begin{figure}[h!] 
		\centering
		\includegraphics[width = 0.5\textwidth]{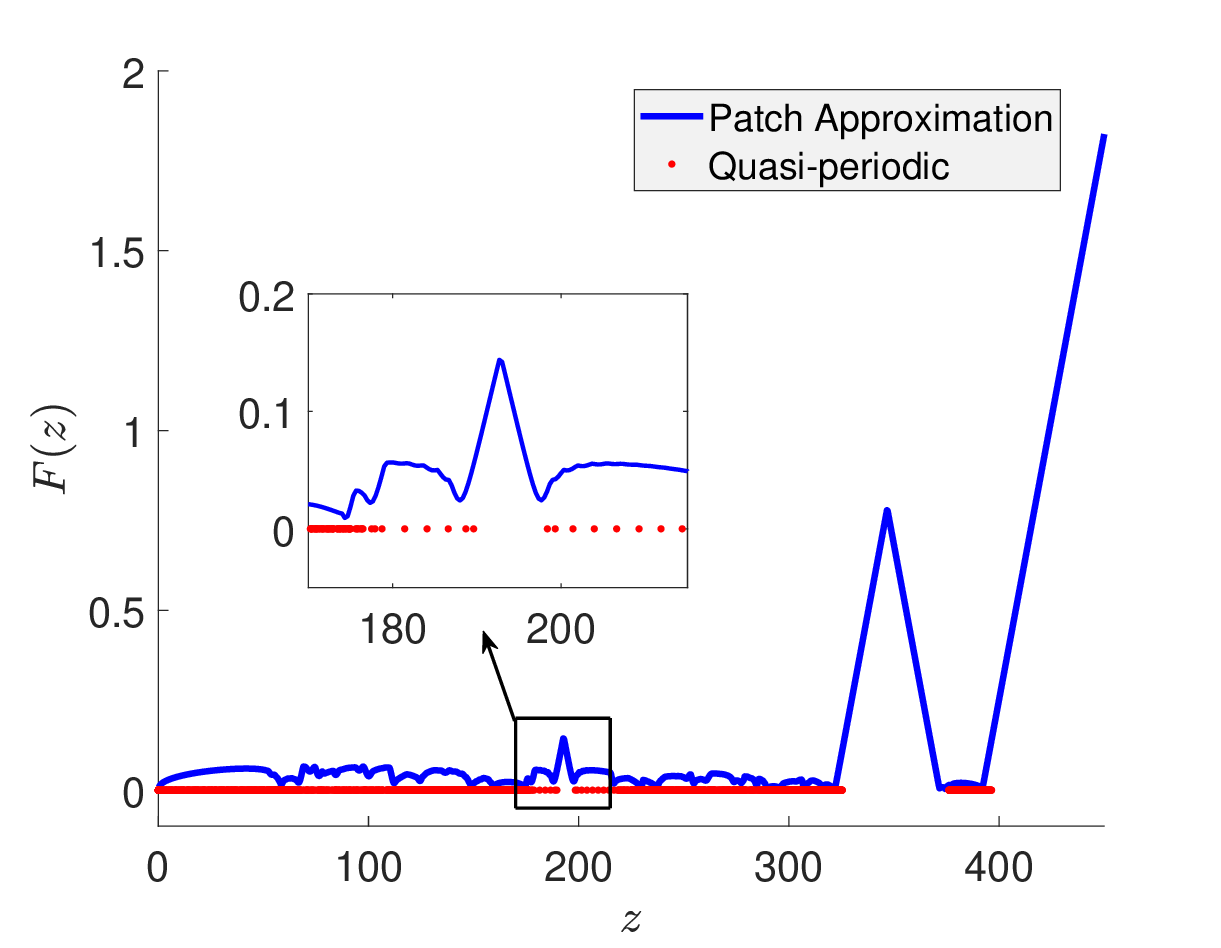}\hspace{-0.5cm}
		\includegraphics[width = 0.5\textwidth]{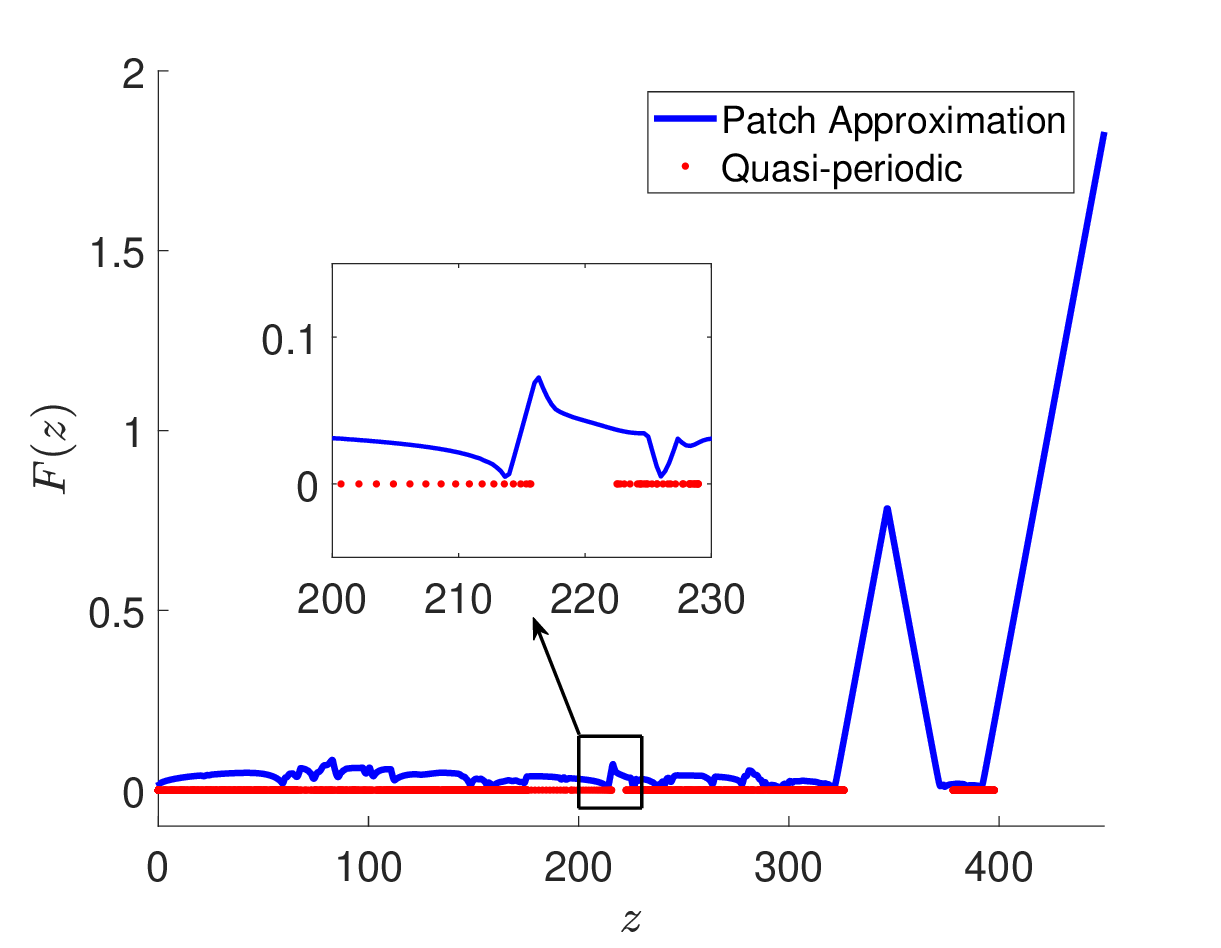}
		\caption{Comparison between the supercell method and the patch approximation for type-I interface and type-II interface. The function $F_{\mathcal{N}}(z)$ is shown together with the quasi-periodic spectrum for each $\alpha$ projected onto the $z$-axis for $z\in\mathbb{R}$, with pseudo-modes being removed. Left panel: type-I interface. Right panel: type-II interface.}
		\label{fig:compare_spectrum}
	\end{figure}

	\crefalias{section}{appendix}
	\crefalias{subsection}{appendix}
	\appendix
	\section{Exterior Dirichlet-to-Neumann map}
	\label{app_d2n}
	
	In this appendix, we prove Proposition \ref{prop_d2n_map}. The idea is to express $\mathcal{T}^{z}$ in terms of the resolvent of the exterior Dirichlet Laplacian. Specifically, we introduce the following operator:
	\begin{equation*}
		\mathcal{L}^{\mathrm{ext}}: H_0^1(\Omega)\subset L^2(\Omega)\to H^{-1}(\Omega),\quad \mathcal{L}^{\mathrm{ext}}u:=-\Delta u.
	\end{equation*}
	In the first part of the proof, we show that there exists $z_0>0$ such that the resolvent $(\mathcal{L}^{\mathrm{ext}}-z)^{-1}$ of $\mathcal{L}^{\mathrm{ext}}$ is well-defined within the disk $\mathbb{D}_0:=\{z\in\mathbb{C}:\, |z|< z_0\}$. Then, in the second part, we show how to write the Dirichlet-to-Neumann map $\mathcal{T}^{z}$ in terms of $(\mathcal{L}^{\mathrm{ext}}-z)^{-1}$, which makes the analyticity and boundedness of $\mathcal{T}^{z}$ clear.
	
	Step 1. To prove that the resolvent $(\mathcal{L}^{\mathrm{ext}}-z)^{-1}$ is well-defined, it suffices to show that
	\begin{equation} \label{eq_d2n_proof_1}
		\int_{\Omega}|\nabla u(\bm{x})|^2\geq z_0 \int_{\Omega}|u(\bm{x})|^2,\quad \forall u\in H_0^1(\Omega),
	\end{equation}
	for some $z_0>0$. Moreover, by the continuity of \eqref{eq_d2n_proof_1} with respect to the $H^1$ norm and the relation $H_0^1(\Omega)=\overline{C_{c}^{\infty}(\Omega)}^{H^1}$, \eqref{eq_d2n_proof_1} is equivalent to
	\begin{equation} \label{eq_d2n_proof_2}
		\int_{\Omega}|\nabla u(\bm{x})|^2\geq z_0 \int_{\Omega}|u(\bm{x})|^2,\quad \forall u\in C_{c}^{\infty}(\Omega).
	\end{equation}
	The validity of \eqref{eq_d2n_proof_2} is rooted in our assumptions on the geometry of resonators, \emph{i.e.}, Assumption \ref{asmp_geometry_assumption}, especially condition \eqref{eq_geometry_assumption_3}. In fact, by the partition of space in conditions \eqref{eq_geometry_assumption_1}-\eqref{eq_geometry_assumption_2}, we only need to prove \eqref{eq_d2n_proof_2} within any translation of polygons:
	\begin{equation} \label{eq_d2n_proof_3}
		\int_{\Omega\cap (Y^{(i)}+\bm{n})}|\nabla u(\bm{x})|^2\geq z_0^{(i)} \int_{\Omega\cap (Y^{(i)}+\bm{n})}|u(\bm{x})|^2,\quad \forall u\in C_{c}^{\infty}(\Omega),\, \bm{n}\in \Sigma^{(i)},
	\end{equation}
	where $z_0^{(i)}>0$ depends only on $i\in\{1,2,\cdots,M_{\mathrm{poly}}\}$. This is where assumption \eqref{eq_geometry_assumption_3} comes into play. In fact, since $(Y^{(i)}+\bm{n})\cap \mathcal{D}\neq \emptyset$, we can select a resonator $D_{\sigma}$ intersecting $Y^{(i)}+\bm{n}$ and evaluate $u(\bm{x})$ using the Newton-Leibniz formula starting from the boundary $\partial D_{\sigma}$. Recalling that $u$ vanishes on $\partial D_{\sigma}$, one can then prove the Poincaré-type estimate \eqref{eq_d2n_proof_3} using a standard argument based on the Schwarz inequality. We omit the details here, which follow the same lines as in \cite[Proposition 2.2, Step 1]{qiu2025nonlinear}.
	
	Step 2. With the result in Step 1 in hand, the uniqueness of a solution to \eqref{eq_d2n_def} within $z\in \mathbb{K}$ is obvious. Now, we explicitly construct a solution to \eqref{eq_d2n_def}, which then shows that the Dirichlet-to-Neumann map $\mathcal{T}^{z}$ is well-defined. Denote the trace operator from $H^1(\Omega)$ to $ H^{1/2}(\partial \Omega)$ by $\mathfrak{R}$. We can construct a right inverse $\mathfrak{E}$ of $\mathfrak{R}$ (lifting operator) as follows:
	\begin{enumerate}
		\item[(i)] take an open neighborhood $O^{(i)}$ of $D^{(i)}$ for each $i$ such that $D^{(i)}\Subset O^{(i)}$ and $O^{(i)}\cap O^{(j)}=\emptyset$ if $i\neq j$;
		\item[(ii)] construct a lifting map $\mathfrak{E}^{(i)}$ from $ H^{1/2}(\partial D^{(i)})$ into $H_0^1(O^{(i)}\backslash D^{(i)})$ following a patching argument, \emph{e.g.}, as in \cite[Theorem 18.18]{Leoni2024};
		\item[(iii)] repeat the same procedure for each translation of the resonator $D_{\sigma}$, and denote the resulting lifting map by $\mathfrak{E}_{\sigma}$. In particular, by  assumption \eqref{eq_geometry_assumption_4}, we take the range of $\mathfrak{E}_{\sigma}$ to be non-overlapping if $\sigma\neq \sigma'$;
		\item[(iv)] add the resulting maps and obtain $\mathfrak{E}:=\sum_{\sigma\in \Sigma}\mathfrak{E}_{\sigma}\circ \mathbbm{1}_{\partial \mathcal{D}_{\sigma}}$.
	\end{enumerate}
	With these preparations, we are now ready to construct the solution to \eqref{eq_d2n_def} using a standard lifting argument. We let
	\begin{equation*}
		f_{z,\phi}:=(\Delta +z)(\mathfrak{E}\phi)\in H^{-1}(\Omega).
	\end{equation*}
	Then we set $g_{z,\phi}:=(\mathcal{L}^{\mathrm{ext}}-z)^{-1}f_{z,\phi}$, which uniquely solves the following problem: 
	\begin{equation*}
		\left\{
		\begin{aligned}
			&-\Delta g_{z,\phi}-z g_{z,\phi}=f_{z,\phi} \quad \text{in } \Omega, \\
			&g_{z,\phi}\big|_{+}=0 \quad \text{on } \partial \Omega,
		\end{aligned}
		\right.
	\end{equation*}
	since the resolvent is well-defined for $\{z:|z|\leq z_0\}$.
	Finally, we let
	\begin{equation} \label{eq_d2n_proof_4}
		u_{z,\phi}:=\mathfrak{E}\phi +g_{z,\phi}=\mathfrak{E}\phi +(\mathcal{L}^{\mathrm{ext}}-z)^{-1}f_{z,\phi},
	\end{equation}
	and see that $u_{z,\phi}$ solves \eqref{eq_d2n_def}.
	
	Step 3. A direct consequence of \eqref{eq_d2n_proof_4} is the following expression for the Dirichlet-to-Neumann map:
	\begin{equation} \label{eq_d2n_proof_5}
		\mathcal{T}^{z}:\quad
		\phi\mapsto \frac{\partial}{\partial\nu}\Big|_{\partial \mathcal{D}}\Big(\mathfrak{E}\phi +(\mathcal{L}^{\mathrm{ext}}-z)^{-1}f_{z,\phi} \Big) .
	\end{equation}
	This, together with the analyticity of the resolvent $(\mathcal{L}^{\mathrm{ext}}-z)^{-1}$ and $f_{z,\phi}$ with respect to $z$, gives the boundedness and analyticity of $\mathcal{T}^{z}$.
	
	\begin{remark} \label{rmk_exp_decay}
		Formula \eqref{eq_d2n_proof_5}, together with the Combes-Thomas estimate of resolvents associated with self-adjoint operators (see, \emph{e.g.}, \cite[Theorem 11.2]{Kirsch2007} or \cite[Proposition 3.1]{Qiu2025}), directly indicates the  off-diagonal exponential decay of the capacitance operator \eqref{eq_cap_operator_def}:
		\begin{equation*}
			\mathcal{C}(\sigma,\sigma')\lesssim \eu^{-\beta\operatorname{dist}(D_{\sigma},D_{\sigma'})},
		\end{equation*}
		where $\beta>0$ is independent of $\sigma,\sigma'$. The key to this point, as is clear in the proof of Proposition \ref{prop_d2n_map}, is that the spectrum of the exterior Dirichlet Laplacian is bounded below by a positive constant, or more fundamentally, the geometric assumption in \eqref{eq_geometry_assumption_3}. In \Cref{apsec:expdecay}, we provide a finer estimate of the off-diagonal decay of the capacitance operator using a heat kernel estimate.
	\end{remark}

	\section{Inverse of the sesquilinear form $\mathfrak{a}$}
	\label{app_inverse_of_form}

	In this appendix, we prove that the left side of \eqref{eq_cont_to_disc_resolvent_converge} is well-defined.
	\begin{proposition} \label{prop_form_resolvent}
		For any $z\notin\mathbb{R}$ with $|\Im z|<1$, there exists $C=C(|z|)>0$ such that, if
		\begin{equation} \label{eq_form_resolvent_delta_cond}
			\delta\in (0,C|\Im z|),
		\end{equation}
		then the variational problem $\mathfrak{a}(u,v;\delta z,\delta)=(f,v)_{L^2(\mathcal{D})}$ has a unique solution $u\in H^1(\mathcal{D})$ for any $f\in L^2(\mathcal{D})$.
	\end{proposition}
	We briefly remark on the main idea of proof before showing the details. Recall the definition of $\mathfrak{a}$:
	\begin{equation*}
		\mathfrak{a}(u,v;\delta z,\delta)
		=(\nabla u,\nabla v)_{(L^2(\mathcal{D}))^2} -\delta z(u,v)_{L^2(\mathcal{D})} -\delta (\mathcal{T}^{\delta z}[u],v)_{L^2(\partial \mathcal{D})} .
	\end{equation*}
	If $\mathcal{T}^{\delta z}$ is replaced by $\mathcal{T}^{0}$, then the conclusion follows directly from the assumption $z\notin \mathbb{R}$ by noting that the operator realization of the form $(\nabla u,\nabla v)_{(L^2(\mathcal{D}))^2}  -\delta (\mathcal{T}^{0}[u],v)_{L^2(\partial \mathcal{D})}$ is self-adjoint because $(\mathcal{T}^{0}[u],v)_{L^2(\partial \mathcal{D})}$ is symmetric. In addition, in that case, the upper bound in \eqref{eq_form_resolvent_delta_cond} is not required. Nevertheless, the $\delta$-dependence of $\mathcal{T}^{\delta z}$ can be controlled. Indeed, the imaginary part of $\delta (\mathcal{T}^{\delta z}[u],u)_{L^2(\partial \mathcal{D})}$ arises entirely from the difference $\mathcal{T}^{\delta z}-\mathcal{T}^{0}$ and is of higher order than $\Im (\delta z)$:
	\begin{equation*}
		\Im \Big[\delta (\mathcal{T}^{\delta z}[u],u)_{L^2(\partial \mathcal{D})}\Big]
		=\delta \Im\Big[ \big((\mathcal{T}^{\delta z}-\mathcal{T}^{0})[u],u\big)_{L^2(\partial \mathcal{D})}\Big] = \mathcal{O}(\delta^2),
	\end{equation*}
	thanks to the analyticity of $\mathcal{T}^{z}$. Hence, this imaginary part can be absorbed into $\Im (\delta z)$ by appropriately controlling the size of $\delta$, as indicated in \eqref{eq_form_resolvent_delta_cond}, and then the invertibility of $\mathfrak{a}(u,v;\delta z,\delta)$ follows.
	
	\begin{proof}[Proof of Proposition \ref{prop_form_resolvent}]
		Write $z=|z|\eu^{\iu\phi}$ and suppose that $\phi\in (0,\frac{\pi}{2})$ without loss of generality. Define
		\begin{equation*}
			\begin{aligned}
				\mathfrak{b}(u,v;\delta z,\delta)&=\iu\eu^{\iu\theta}\mathfrak{a}(u,v;\delta z,\delta) \\
				&=\iu\eu^{\iu\theta}(\nabla u,\nabla v)_{(L^2(\mathcal{D}))^2} -\iu\delta z \eu^{\iu\theta} (u,v)_{L^2(\mathcal{D})} -\iu\delta \eu^{\iu\theta} (\mathcal{T}^{\delta z}[u],v)_{L^2(\partial \mathcal{D})}
			\end{aligned}
		\end{equation*}
		with $\theta\in (-\frac{\pi}{2},\frac{\pi}{2})$ to be determined. Then the invertibility of $\mathfrak{a}$ is equivalent to that of $\mathfrak{b}$. Therefore, it suffices to show that
		\begin{equation} \label{eq_form_resolvent_proof_1}
			\Re \mathfrak{b}(u,u;\delta z,\delta)\geq c\|u\|_{H^1(\mathcal{D})}^2,\quad \forall u\in H^1(\mathcal{D}),
		\end{equation}
		for some $c>0$. One calculates that
		\begin{equation} \label{eq_form_resolvent_proof_2}
			\begin{aligned}
				\Re \mathfrak{b}(u,u;\delta z,\delta)&=-\sin(\theta)\int_{\mathcal{D}}|\nabla u|^2+\delta|z|\sin(\phi+\theta)\int_{\mathcal{D}}|u|^2 \\
				&\quad +\delta\sin(\theta)\Re (\mathcal{T}^{\delta z}[u],u)_{L^2(\partial \mathcal{D})}+\delta\cos(\theta)\Im (\mathcal{T}^{\delta z}[u],u)_{L^2(\partial \mathcal{D})}.
			\end{aligned}
		\end{equation}
		Since $(\mathcal{T}^{0}[u],u)\in\mathbb{R}$ and $\mathcal{T}^{z}$ is analytic in $z$, 
		we have
		\begin{equation} \label{eq_form_resolvent_proof_7}
			\big|\Re (\mathcal{T}^{\delta z}[u],u)_{L^2(\partial \mathcal{D})}\big|\leq \|\mathcal{T}^{\delta z}\|_{\mathcal{B}(H^{1/2}(\partial \mathcal{D}), H^{-1/2}(\partial \mathcal{D}))}\|u\|^2_{H^{1/2}(\partial \mathcal{D})} \leq C_1\|u\|^2_{H^1(\mathcal{D})}
		\end{equation}
		and 
		\begin{equation} \label{eq_form_resolvent_proof_8}
			\big|\Im (\mathcal{T}^{\delta z}[u],u)_{L^2(\partial \mathcal{D})}\big|
			=\big|\Im \big((\mathcal{T}^{\delta z}-\mathcal{T}^{0})[u],u\big)_{L^2(\partial \mathcal{D})}\big|\leq C_2 \delta\|u\|^2_{H^1(\mathcal{D})},
		\end{equation}
		where $C_1,C_2>0$ are independent of $z$ and $\delta$, given that $\delta \, |z|$ is sufficiently small (this poses a restriction on $\delta$ which only depends on $|z|$). Hence, the right side of \eqref{eq_form_resolvent_proof_2} is estimated as
		\begin{equation} \label{eq_form_resolvent_proof_3}
			\begin{aligned}
				\Re \mathfrak{b}(u,u;\delta z,\delta)&\geq \Big[ -\sin(\theta)+C_1\delta \sin(\theta)-C_2\delta^2\cos(\theta) \Big]\int_{\mathcal{D}}|\nabla u|^2 \\
				&\quad +\delta\Big[ |z|\sin(\phi+\theta)+C_1\sin(\theta)-C_2\delta\cos(\theta) \Big]\int_{\mathcal{D}}| u|^2 .
			\end{aligned}
		\end{equation}
		We set $\delta<\frac{1}{2C_1}$, and select $\theta<0$ with 
		\begin{equation} \label{eq_form_resolvent_proof_4}
			|\theta|<\min\{\frac{1}{2}\phi,\sin^{-1}\big(\frac{1}{4C_1}\Im z\big) \},
		\end{equation}
		which guarantees that
		\begin{equation*}
			-\sin(\theta)+C_1\delta \sin(\theta)\geq \frac{1}{2}\sin(-\theta)>0,
		\end{equation*}
		and
		\begin{equation*}
			|z|\sin(\phi+\theta)+C_1\sin(\theta)\geq \frac{1}{2}|z|\sin(\frac{1}{2}\phi)>0.
		\end{equation*}
		Thus, \eqref{eq_form_resolvent_proof_3} is further estimated as
		\begin{equation} \label{eq_form_resolvent_proof_5}
			\begin{aligned}
				\Re \mathfrak{b}(u,u;\delta z,\delta)&\geq \Big[ \frac{1}{2}\sin(-\theta)-C_2\delta^2\cos(\theta) \Big]\int_{\mathcal{D}}|\nabla u|^2 \\
				&\quad +\delta\Big[ \frac{1}{2}|z|\sin(\frac{1}{2}\phi)-C_2\delta\cos(\theta) \Big]\int_{\mathcal{D}}| u|^2 .
			\end{aligned}
		\end{equation}
		By further restricting $\delta$ so that
		\begin{equation} \label{eq_form_resolvent_proof_6}
			\frac{1}{2}\sin(-\theta)-C_2\delta^2\geq \frac{1}{4}\sin(-\theta),\quad \frac{1}{2}|z|\sin(\frac{1}{2}\phi)-C_2\delta\geq \frac{1}{4}|z|\sin(\frac{1}{2}\phi),
		\end{equation}
		which concludes the proof of \eqref{eq_form_resolvent_proof_1}. We note that, in light of \eqref{eq_form_resolvent_proof_4}, \eqref{eq_form_resolvent_proof_6} poses a restriction $\delta<C|\Im z|$ with $C=C(|z|)$ depending only on the amplitude $|z|$, which actually absorbs the restrictions that have emerged in \eqref{eq_form_resolvent_proof_7} and \eqref{eq_form_resolvent_proof_8} since we have assumed that $|\Im z|<1$.

	\end{proof}
	
	\section{Norm resolvent convergence: Proof of Theorem \ref{thm_cont_to_disc_resolvent_converge}} \label{sec_resolvent_convergence}
	
	In this appendix, we prove Theorem \ref{thm_cont_to_disc_resolvent_converge} following the idea outlined in Section \ref{sec_heuristics}. Throughout, we work in the subwavelength scaling $\lambda=\delta z$ with $z\in \mathbb{K}$. A key to the proof is the regularity estimate of the solution to the variational problem $\mathfrak{a}(u,v;\delta z,\delta)=(f,v)_{L^2(\mathcal{D})}$, presented in \Cref{sec_regularity_estimate}, with which we prove Theorem \ref{thm_cont_to_disc_resolvent_converge} in \Cref{sec_resolvent_convergence_proof}.
	
	\subsection{Regularity estimate of the solution to the variational problem}
	\label{sec_regularity_estimate}
	
	The following regularity estimate will be frequently applied.
	\begin{lemma} \label{lem_regularity_sesq_solution}
		Let $\mathbb{K}\subset \mathbb{C}$ be a compact region such that $\mathbb{K}\cap \mathbb{R}=\emptyset$. Then, there exists $\delta_0=\delta_0(|\mathbb{K}|)>0$ (depending only on $|\mathbb{K}|:=\sup_{z\in\mathbb{K}}|z|$) such that, for any $z\in \mathbb{K}$, $0<\delta<\delta_0$ and $f\in L^2(\mathcal{D})$, the solution $u$ to the variational problem $\mathfrak{a}(u,v;\delta z,\delta)=(f,v)_{L^2(\mathcal{D})}$ satisfies the following estimates:
		\begin{equation} \label{eq_regularity_sesq_solution_perp}
			\|\mathcal{P}_{\perp}u\|_{H^1(\mathcal{D})}\le  C_1\Big[\|\mathcal{P}_{\perp}f\|_{L^2(\mathcal{D})}+\delta \|\mathcal{P} u\|_{H^1(\mathcal{D})}\Big],
		\end{equation}
		and 
		\begin{gather}\label{eq_regularity_sesq_solution_parallel}
			\begin{aligned}
				\|\mathcal{P} u\|_{H^1(\mathcal{D})}&\le  \frac{\|(\mathcal{C}-z)^{-1}\|_{\mathcal{B}(\ell^2(\Sigma))}}{\delta\big(1-C_2\delta \|(\mathcal{C}-z)^{-1}\|_{\mathcal{B}(\ell^2(\Sigma))} \big)}\|\mathcal{P} f\|_{L^2(\mathcal{D})}\\
				&\quad+\frac{C_2\|(\mathcal{C}-z)^{-1}\|_{\mathcal{B}(\ell^2(\Sigma))}}{1-C_2\delta \|(\mathcal{C}-z)^{-1}\|_{\mathcal{B}(\ell^2(\Sigma))}}\|\mathcal{P}_{\perp} u\|_{H^1(\mathcal{D})},	
			\end{aligned} 
		\end{gather}
		where $C_i=C_i(|\mathbb{K}|)>0$ depends only on $|\mathbb{K}|$ ($i=1,2$).
	\end{lemma}
	
	Intuitively speaking, the idea of Lemma \ref{lem_regularity_sesq_solution} is to bound the component $\mathcal{P}_{\perp} u$ using the Poincaré inequality, while the constant part $\mathcal{P} u$ is bounded using the resolvent of the capacitance operator.
	
	\begin{proof}[Proof of Lemma \ref{lem_regularity_sesq_solution}]
		We first prove \eqref{eq_regularity_sesq_solution_perp}. Taking $v=\mathcal{P}_{\perp} u$ in $\mathfrak{a}(u,v;\delta z,\delta)=(f,v)_{L^2(\mathcal{D})}$ and using the orthogonality relation
		\[
		(w_1,\mathcal{P}_{\perp}w_2)_{L^2(\mathcal{D})}=(\mathcal{P}_{\perp} w_1,\mathcal{P}_{\perp}w_2)_{L^2(\mathcal{D})} ,\quad \forall w_1,w_2\in H^1(\mathcal{D})),
		\] 
		we obtain
		\begin{equation} \label{eq_regularity_sesq_solution_proof_1}
			\begin{aligned}
				&\int_{\mathcal{D}}|\nabla\mathcal{P}_{\perp} u|^2\,\du x -\delta z\int_{\mathcal{D}}|\mathcal{P}_{\perp} u|^2\,\du x -\delta (\mathcal{T}^{\delta z}[\mathcal{P}_{\perp} u],\mathcal{P}_{\perp} u)_{L^2(\partial \mathcal{D})} \\
				&=(\mathcal{P}_{\perp}f,\mathcal{P}_{\perp}u)_{L^2(\mathcal{D})}+\delta (\mathcal{T}^{\delta z}[\mathcal{P} u],\mathcal{P}_{\perp} u)_{L^2(\partial \mathcal{D})} .
			\end{aligned}
		\end{equation}
		Note that the left side is coercive for small $\delta$. Since $\mathcal{D}$ is the union of finitely many types of resonators, the Poincaré inequality on each $D_{\sigma}$ and summation over $\sigma\in \Sigma$ yield
		\begin{equation} \label{eq_regularity_sesq_solution_proof_2}
			\|\nabla\mathcal{P}_{\perp} u\|^2_{L^2(\mathcal{D})} \ge  C\|\mathcal{P}_{\perp} u\|_{H^1(\mathcal{D})}^2,
		\end{equation}
		for some $C>0$ depending only on the geometry. The constant is uniform because only finitely many reference resonator shapes occur. Next, let
		\[
		M_0:=\sup_{|z|<z_0}\|\mathcal{T}^{z}\|_{\mathcal{B}(H^{1/2}(\partial\mathcal{D}), H^{-1/2}(\partial\mathcal{D}))}
		\|\mathfrak{R}\|_{\mathcal{B}(H^1(\mathcal{D}), H^{1/2}(\partial\mathcal{D}))}^2,
		\]
		where $\mathfrak{R}:H^1(\mathcal{D})\to H^{1/2}(\partial \mathcal{D})$ denotes the trace operator and $z_0$ is given in Proposition \ref{prop_d2n_map}. We choose $\delta_0$ satisfying the following restrictions:
		\begin{gather} \label{eq_regularity_sesq_solution_proof_3}
			\begin{aligned}
				\delta_0<1,\quad \delta_0|\mathbb{K}|\le \min\{\frac{C}{4},z_0\},\quad \delta_0M_0 \le  \frac{C}{4}.
			\end{aligned}
		\end{gather}
		Hence, for any $z\in\mathbb{K}$ and $\delta<\delta_0$, using \eqref{eq_regularity_sesq_solution_proof_2}, the boundedness of $\mathcal{T}^{\delta z}$, and the trace estimate, we have
		\begin{equation} \label{eq_regularity_sesq_solution_proof_4}
			\begin{aligned}
				&\bigg|\int_{\mathcal{D}}|\nabla\mathcal{P}_{\perp} u|^2\,\du x -\delta z\int_{\mathcal{D}}|\mathcal{P}_{\perp} u|^2\, \du x -\delta (\mathcal{T}^{\delta z}[\mathcal{P}_{\perp} u],\mathcal{P}_{\perp} u)_{L^2(\partial \mathcal{D})}\bigg|\\ &\quad\ge  C\|\mathcal{P}_{\perp} u\|_{H^1(\mathcal{D})}^2 - \delta |\mathbb{K}|\| \mathcal{P}_{\perp} u \|_{H^1(\mathcal{D})}^2 - \delta M_0 \|\mathcal{P}_{\perp} u\|_{H^1(\mathcal{D})}^2\ge \frac{C}{2}\|\mathcal{P}_{\perp} u\|_{H^1(\mathcal{D})}^2 .
			\end{aligned}
		\end{equation}
		This leads to 
		\[
		\frac{C}{2}\|\mathcal{P}_{\perp}u\|_{H^1(\mathcal{D})}^2
		\le  \|\mathcal{P}_{\perp}f\|_{L^2(\mathcal{D})}\|\mathcal{P}_{\perp}u\|_{H^1(\mathcal{D})}
		+\delta M_0 \|\mathcal{P} u\|_{H^1(\mathcal{D})}\|\mathcal{P}_{\perp}u\|_{H^1(\mathcal{D})},
		\]
		and hence gives \eqref{eq_regularity_sesq_solution_perp} after division by $\|\mathcal{P}_{\perp}u\|_{H^1(\mathcal{D})}$ when this quantity is nonzero.
		The constant $C_1$ is given by
		\begin{equation*}
			C_1=\frac{2}{C}\max\big\{1,M_0\big\}.
		\end{equation*}
		
		We next prove \eqref{eq_regularity_sesq_solution_parallel}. The idea is similar to the first part, but now we take the test function in the space $\operatorname{Ran}\mathcal{P}$. Specifically, taking $v=\mathbbm{1}_{D_{\sigma}}$ in $\mathfrak{a}(u,v;\delta z,\delta)=(f,v)_{L^2(\mathcal{D})}$, we have
		\begin{equation} \label{eq_regularity_sesq_solution_proof_5}
			\begin{aligned}
				-\delta z&\int_{D_{\sigma}}\mathcal{P} u\,\du x -\delta \int_{\partial D_{\sigma}}\mathcal{T}^{0}[\mathcal{P} u]\,\du \sigma(y) \\
				&=\int_{D_{\sigma}}\mathcal{P} f\,\du x +\delta \int_{\partial D_{\sigma}}\mathcal{T}^{\delta z}[\mathcal{P}_{\perp} u]\,\du \sigma(y) +\delta \int_{\partial D_{\sigma}}(\mathcal{T}^{\delta z}-\mathcal{T}^{0})[\mathcal{P} u]\,\du \sigma(y).
			\end{aligned}
		\end{equation}
		Recalling the definition of the capacitance operator and using that \eqref{eq_regularity_sesq_solution_proof_5} holds for any $\sigma\in \Sigma$, we rewrite it as 
		\begin{equation} \label{eq_regularity_sesq_solution_proof_6}
			\delta(\mathcal{C}-z)Pu=P\mathcal{P} f +\delta P_{\partial}\big(\mathcal{T}^{\delta z}[\mathcal{P}_{\perp} u] +(\mathcal{T}^{\delta z}-\mathcal{T}^{0})[\mathcal{P} u]\big),
		\end{equation}
		where the projection onto the boundary $P_{\partial}:H^{-1/2}(\partial\mathcal{D})\to \ell^2(\Sigma)$ is defined by
		\begin{equation*}
			(P_{\partial }u)(\sigma):=\frac{1}{|\partial D_{\sigma}|}\int_{\partial D_{\sigma}}u\,\du \sigma(y) .
		\end{equation*}
		Since there are only finitely many resonator types, $P_{\partial}$ is bounded. Moreover, by Proposition \ref{prop_d2n_map}, the map $z\mapsto \mathcal{T}^{z}$ is analytic for $|z|<z_0$, there exists a constant $C_2=C_2(|\mathbb{K}|)$ for sufficiently small $\delta_0$ such that
		\begin{gather*}
			\begin{aligned}
				\sup_{z\in\mathbb{K},\,0<\delta<\delta_0}
				\|\mathcal{T}^{\delta z}\|_{H^{1/2}(\partial\mathcal{D})\to H^{-1/2}(\partial\mathcal{D})}
				\le  C_2,\\
				\sup_{z\in\mathbb{K},\,0<\delta<\delta_0}
				\|\mathcal{T}^{\delta z}-\mathcal{T}^{0}\|_{H^{1/2}(\partial\mathcal{D})\to H^{-1/2}(\partial\mathcal{D})}
				\le  C_2\delta.
			\end{aligned}
		\end{gather*}
		Using these bounds together with the trace estimate, we obtain from \eqref{eq_regularity_sesq_solution_proof_6} that
		\begin{equation} \label{eq_regularity_sesq_solution_proof_7}
			\begin{aligned}
				&\|P\mathcal{P} f+\delta P_{\partial}\big(\mathcal{T}^{\delta z}[\mathcal{P}_{\perp}u]+(\mathcal{T}^{\delta z}-\mathcal{T}^{0})[\mathcal{P} u]\big)\|_{\ell^2(\Sigma)} \\
				&\le  \|\mathcal{P} f\|_{L^2(\mathcal{D})}
				+ C_2\delta\big(\|\mathcal{P}_{\perp}u\|_{H^1(\mathcal{D})}+\delta\|\mathcal{P} u\|_{H^1(\mathcal{D})}\big).
			\end{aligned}
		\end{equation}
		Since $\mathcal{C}$ is self-adjoint, we have for $z\in \rho(\mathcal{C})$,
		\begin{equation} \label{eq_regularity_sesq_solution_proof_8}
			\|\delta(\mathcal{C}-z)Pu\|_{\ell^2(\Sigma)}
			\ge  \delta\|(\mathcal{C}-z)^{-1}\|_{\mathcal{B}(\ell^2(\Sigma))}^{-1}\|Pu\|_{\ell^2(\Sigma)}
			=
			\delta\|(\mathcal{C}-z)^{-1}\|_{\mathcal{B}(\ell^2(\Sigma))}^{-1}\|\mathcal{P} u\|_{H^1(\mathcal{D})},
		\end{equation}
		where, in the last step, we have used that $\mathcal{P} u$ is piecewise constant on each $D_{\sigma}$ and $|D_{\sigma}|=1$ for all $\sigma\in\Sigma$. Combining \eqref{eq_regularity_sesq_solution_proof_6}, \eqref{eq_regularity_sesq_solution_proof_7}, and \eqref{eq_regularity_sesq_solution_proof_8}, we arrive at
		\[
		\delta\|(\mathcal{C}-z)^{-1}\|^{-1}\|\mathcal{P} u\|_{H^1(\mathcal{D})}
		\le  \|\mathcal{P} f\|_{L^2(\mathcal{D})}
		+ C_2\delta\|\mathcal{P}_{\perp}u\|_{H^1(\mathcal{D})}
		+ C_2\delta^2\|\mathcal{P} u\|_{H^1(\mathcal{D})}.
		\]
		Multiplying both sides by $\|(\mathcal{C}-z)^{-1}\|/\delta$ and rearranging the terms gives
		\[
		\big(1-C_2\delta\|(\mathcal{C}-z)^{-1}\|\big)\|\mathcal{P} u\|_{H^1(\mathcal{D})}
		\le 
		\frac{\|(\mathcal{C}-z)^{-1}\|}{\delta}\|\mathcal{P} f\|_{L^2(\mathcal{D})}
		+
		C_2\|(\mathcal{C}-z)^{-1}\|\|\mathcal{P}_{\perp}u\|_{H^1(\mathcal{D})},
		\]
		which is exactly \eqref{eq_regularity_sesq_solution_parallel}.
		
	\end{proof}

	\subsection{Proof of Theorem \ref{thm_cont_to_disc_resolvent_converge}}
	\label{sec_resolvent_convergence_proof}
	
	Step 1. We first prove the norm convergence of the resolvents on $\operatorname{Ran}\mathcal{P}_{\perp}$. Suppose that $u\in H^1(\mathcal{D})$ solves $\mathfrak{a}(u,v;\delta z,\delta)=(f,v)_{L^2(\mathcal{D})}$ with $f\in \operatorname{Ran}\mathcal{P}_{\perp}$. Then, substituting \eqref{eq_regularity_sesq_solution_parallel} into \eqref{eq_regularity_sesq_solution_perp} with $\mathcal{P} f=0$ yields
	\begin{equation*}
		\|\mathcal{P}_{\perp}u\|_{H^1(\mathcal{D})}\le  C_1\Big[\|f\|_{L^2(\mathcal{D})}+ 
		\frac{C_2\delta\|(\mathcal{C}-z)^{-1}\|_{\mathcal{B}(\ell^2(\Sigma))}}{1-C_2\delta \|(\mathcal{C}-z)^{-1}\|_{\mathcal{B}(\ell^2(\Sigma))}}\|\mathcal{P}_{\perp} u\|_{H^1(\mathcal{D})} \Big].
	\end{equation*}
	We then derive
	\begin{equation} \label{eq_cont_to_disc_resolvent_converge_proof_1}
		\|\mathcal{P}_{\perp}u\|_{H^1(\mathcal{D})}\le  2C_1  \|f\|_{L^2(\mathcal{D})},
	\end{equation}
	given that
	\begin{equation} \label{eq_cont_to_disc_resolvent_converge_proof_2}
		\delta<\min\Big(\frac{1}{2(C_1+1)C_2\|(\mathcal{C}-z)^{-1}\|_{\mathcal{B}(\ell^2(\Sigma))}},\delta_0\Big),
	\end{equation}
	where $\delta_0$ is the small parameter given in \Cref{lem_regularity_sesq_solution}. By \eqref{eq_cont_to_disc_resolvent_converge_proof_1}, it is clear that
	\begin{equation} \label{eq_cont_to_disc_resolvent_converge_proof_3}
		\|\delta\mathcal{P}_{\perp}\mathcal{R}(\delta z,\delta)\mathcal{P}_{\perp}\|_{\mathcal{B}(L^2(\mathcal{D}))}
		\le   2C_1\delta \to 0,
	\end{equation}
	as $\delta\to 0$, which concludes the norm convergence of the resolvents on $\operatorname{Ran}\mathcal{P}_{\perp}$.
	
	Step 2. Next, we show the convergence of the off-diagonal components, \emph{i.e.}, $\delta\mathcal{P}_{\perp}\mathcal{R}(\delta z,\delta)\mathcal{P}$ and $\delta\mathcal{P}\mathcal{R}(\delta z,\delta)\mathcal{P}_{\perp}$, as $\delta$ goes to zero. We prove only the first statement, since the second statement can be treated similarly. Suppose that $u\in H^1(\mathcal{D})$ solves $\mathfrak{a}(u,v;\delta z,\delta)=(f,v)_{L^2(\mathcal{D})}$ with $f\in \operatorname{Ran}\mathcal{P}$. Again, using Lemma \ref{lem_regularity_sesq_solution} and recalling that this time we have $\mathcal{P}_{\perp} f=0$, one sees
	\begin{equation*}
		\begin{aligned}
			\|\mathcal{P}_{\perp}u\|_{H^1(\mathcal{D})} 
			&\le   
			\frac{C_1\|(\mathcal{C}-z)^{-1}\|_{\mathcal{B}(\ell^2(\Sigma))}}{1-C_2\delta \|(\mathcal{C}-z)^{-1}\|_{\mathcal{B}(\ell^2(\Sigma))} }\| f\|_{L^2(\mathcal{D})}\\
			&\quad +\frac{C_1C_2\delta\|(\mathcal{C}-z)^{-1}\|_{\mathcal{B}(\ell^2(\Sigma))}}{1-C_2\delta \|(\mathcal{C}-z)^{-1}\|_{\mathcal{B}(\ell^2(\Sigma))}}\|\mathcal{P}_{\perp} u\|_{H^1(\mathcal{D})},
		\end{aligned}
	\end{equation*}
	or, more clearly,
	\begin{equation} \label{eq_cont_to_disc_resolvent_converge_proof_4}
		\|\mathcal{P}_{\perp}u\|_{H^1(\mathcal{D})} \le   
		2C_1\|(\mathcal{C}-z)^{-1}\|_{\mathcal{B}(\ell^2(\Sigma))}\| f\|_{L^2(\mathcal{D})},
	\end{equation}
	given that \eqref{eq_cont_to_disc_resolvent_converge_proof_2} holds. Hence, we conclude by \eqref{eq_cont_to_disc_resolvent_converge_proof_4} that
	\begin{equation} \label{eq_cont_to_disc_resolvent_converge_proof_5}
		\|\delta\mathcal{P}_{\perp}\mathcal{R}(\delta z,\delta)\mathcal{P}\|_{\mathcal{B}(L^2(\mathcal{D}))}
		\le  2C_1\delta\|(\mathcal{C}-z)^{-1}\|_{\mathcal{B}(\ell^2(\Sigma))} \to 0,
	\end{equation}
	as $\delta\to 0$. We can also conclude that 
	\begin{equation} \label{eq_cont_to_disc_resolvent_converge_proof_5p}
		\|\delta\mathcal{P}\mathcal{R}(\delta z,\delta)\mathcal{P}_{\perp}\|_{\mathcal{B}(L^2(\mathcal{D}))}
		\le  2C_1\delta\|(\mathcal{C}-z)^{-1}\|_{\mathcal{B}(\ell^2(\Sigma))} \to 0.
	\end{equation}
	
	Step 3. Finally, we prove that
	\begin{equation} \label{eq_cont_to_disc_resolvent_converge_proof_6}
		\|\delta\mathcal{P}\mathcal{R}(\delta z,\delta)\mathcal{P}-P^{\ast}(\mathcal{C}-z)^{-1}P\|_{\mathcal{B}(L^2(\mathcal{D}))}\le  2C_2(1+2C_1)\delta \|(\mathcal{C}-z)^{-1}\|_{\mathcal{B}(\ell^2(\Sigma))}^2.
	\end{equation}
	As in Step 2, we suppose that $u\in H^1(\mathcal{D})$ solves $\mathfrak{a}(u,v;\delta z,\delta)=(f,v)_{L^2(\mathcal{D})}$ with $f\in \operatorname{Ran}\mathcal{P}$. Recall that we have shown in \eqref{eq_regularity_sesq_solution_proof_6} that $Pu$ satisfies
	\begin{equation*}
		\delta(\mathcal{C}-z)Pu=P f +\delta P_{\partial}\big(\mathcal{T}^{\delta z}[\mathcal{P}_{\perp} u] +(\mathcal{T}^{\delta z}-\mathcal{T}^{0})[\mathcal{P} u]\big).
	\end{equation*}
	Hence, letting $a:=(\mathcal{C}-z)^{-1}Pf$, we know that
	\begin{equation*}
		(\mathcal{C}-z)(\delta Pu -a)=\delta P_{\partial}\big(\mathcal{T}^{\delta z}[\mathcal{P}_{\perp} u] +(\mathcal{T}^{\delta z}-\mathcal{T}^{0})[\mathcal{P} u]\big).
	\end{equation*}
	Using the estimate in \eqref{eq_regularity_sesq_solution_proof_7}, the right side is bounded as
	\begin{equation} \label{eq_cont_to_disc_resolvent_converge_proof_7}
		\|(\mathcal{C}-z)(\delta Pu -a)\|_{\ell^2(\Sigma)}\le  C_2\delta\big(\|\mathcal{P}_{\perp} u\|_{H^1(\mathcal{D})}+\delta \|\mathcal{P} u\|_{H^1(\mathcal{D})}\big).
	\end{equation}
	First, we bound $\mathcal{P} u$ by $\mathcal{P}_{\perp} u$ and $f$ using \eqref{eq_regularity_sesq_solution_parallel}, which produces
	\begin{equation*}
		\begin{aligned}
			&\|(\mathcal{C}-z)(\delta Pu -a)\|_{\ell^2(\Sigma)} \\
			&\le  C_2\delta\big(1+2C_2\delta\|(\mathcal{C}-z)^{-1}\|_{\mathcal{B}(\ell^2(\Sigma))}\big)\|\mathcal{P}_{\perp} u\|_{H^1(\mathcal{D})}\\
			&\quad+2C_2\delta \|(\mathcal{C}-z)^{-1}\|_{\mathcal{B}(\ell^2(\Sigma))}\|\mathcal{P} f\|_{L^2(\mathcal{D})} \\
			&\le  2C_2\delta \|\mathcal{P}_{\perp} u\|_{H^1(\mathcal{D})} + 2C_2\delta \|(\mathcal{C}-z)^{-1}\|_{\mathcal{B}(\ell^2(\Sigma))}\|\mathcal{P} f\|_{L^2(\mathcal{D})},
		\end{aligned}
	\end{equation*}
	where the restriction \eqref{eq_cont_to_disc_resolvent_converge_proof_2} on $\delta$  is applied. Then, we replace $\mathcal{P}_{\perp} u$ with $f$ using \eqref{eq_cont_to_disc_resolvent_converge_proof_4}:
	\begin{equation} \label{eq_cont_to_disc_resolvent_converge_proof_8}
		\|(\mathcal{C}-z)(\delta Pu -a)\|_{\ell^2(\Sigma)}
		\le  2C_2(1+2C_1)\delta \|(\mathcal{C}-z)^{-1}\|_{\mathcal{B}(\ell^2(\Sigma))}\|\mathcal{P} f\|_{L^2(\mathcal{D})}.
	\end{equation}
	Finally, we bound the left side of \eqref{eq_cont_to_disc_resolvent_converge_proof_8} from below as in \eqref{eq_regularity_sesq_solution_proof_8}, which completes the proof of \eqref{eq_cont_to_disc_resolvent_converge_proof_6} by the following identity:
	\begin{equation*}
		\|\delta Pu -a\|_{\ell^2(\Sigma)}=\|P^{\ast}(\delta Pu -a)\|_{L^2(\mathcal{D})}
		=\|\delta \mathcal{P} u -P^{\ast}(\mathcal{C}-z)^{-1}Pf\|_{L^2(\mathcal{D})},
	\end{equation*}
	where we used the normalization $|D_{\sigma}|=1$ for all $\sigma\in \Sigma$.
	
	Summary. By \eqref{eq_cont_to_disc_resolvent_converge_proof_3}, \eqref{eq_cont_to_disc_resolvent_converge_proof_5}, \eqref{eq_cont_to_disc_resolvent_converge_proof_5p} and \eqref{eq_cont_to_disc_resolvent_converge_proof_6}, we conclude that
	\begin{equation*}
		\|\delta\mathcal{R}(\delta z,\delta)-P^{\ast}(\mathcal{C}-z)^{-1}P\|_{\mathcal{B}(L^2(\mathcal{D}))}\le  C\delta \|(\mathcal{C}-z)^{-1}\|_{\mathcal{B}(\ell^2(\Sigma))}^2,
	\end{equation*}
	given that
	\begin{equation*}
		\delta<\delta_1\min\{1,\frac{1}{\|(\mathcal{C}-z)^{-1}\|_{\mathcal{B}(\ell^2(\Sigma))}}\},
	\end{equation*}
	where $C,\delta_1>0$ depend only on $|\mathbb{K}|$. This completes the proof.

	\section{Exponential decay of the off-diagonal elements of the capacitance operator}\label{apsec:expdecay}
	In this appendix, we adopt a heat kernel estimate to estimate the exponential decay of the capacitance operator. 
	Define the bilinear form $\mathfrak{q}_{\Omega}$ by
	\[ \mathfrak{q}_{\Omega} (u,v):= \int_{\Omega} \nabla u \cdot \overline{\nabla v} \, \du x, \quad u,v\in H_0^1(\Omega).\]
	By assumption \eqref{eq_geometry_assumption_3}, the Poincar\'e inequality holds on $\Omega$. Therefore, there exists a positive constant $\lambda_{\Omega}>0$ such that 
	\[ \| \nabla u \|_{L^2(\Omega)} \ge \lambda_{\Omega} \| u \|_{L^2(\Omega)}. \]
	Let the heat kernel $p(t,x;y)$ be given by
	\begin{gather}
		\left\{\begin{aligned}
			\partial_t p - \Delta_x p  &= \delta(t)\delta(x-y),\quad x\in \Omega, t\in[0,\infty),\\
			p &= 0,\quad x\in\partial \Omega.
		\end{aligned}\right.
	\end{gather}
	We have the following heat kernel estimate for $\eta\in(0,1)$, which follows from \cite[Corollary 6.15]{Ouhabaz2005}:
	\[ 0\le p(t,x;y) \le C_{\eta}t^{-1}\eu^{-(1-\eta) \lambda_{\Omega} t}\eu^{-(1-\eta)\|x-y\|^2/(4t)},\quad x,y\in \Omega, t> 0. \]
	Then the Green function on $\Omega$ is given by 
	\begin{equation}
		G(x;y) = \int_{0}^{\infty} p(t,x;y)\, \du t.
	\end{equation}
	Using the identity 
	\[ \int_{0}^{\infty} t^{-1}\eu^{-(1-\eta) \lambda t}\eu^{-(1-\eta)\|x-y\|^2/(4t)}\, \du t = 2K_0(2(1-\eta)\sqrt{\lambda_{\Omega}}\| x-y \|),\quad x\neq y, \]
	where $K_0$ is the zeroth modified Bessel function, we conclude that for $\| x-y \| \ge \mu$ and $\eta\in(0,1)$, there exists a constant $C_{\mu,\eta}$ such that 
	\begin{equation}\label{apeqn:Greenfunc_estimate}
		0\le G(x;y) \le C_{\mu,\eta} \frac{\eu^{-(1-\eta)\sqrt{\lambda_{\Omega}}\| x-y \|}}{\sqrt{\| x-y \|}}.
	\end{equation}
	Therefore, for a given $\sigma\in\Sigma $, the function $V_{\sigma}$ defined in \eqref{eqn:HarmonicFunc_full} is given by 
	\begin{equation}\label{apeqn:VGreen_rep}
		V_{\sigma}(x) = \int_{\partial D_{\sigma}} \frac{\partial G(x;y)}{\partial \nu_{y}}\, \du \sigma(y),\quad x\in \Omega.
	\end{equation}
	For $\sigma'\neq \sigma$, the geometric assumption \eqref{eq_geometry_assumption_3} ensures the following gradient estimate at the boundary; see \cite{Gilbarg2001}.
	\begin{lemma}\label{aplem:normal_deri}
		Suppose that $v$ is a harmonic function in a neighborhood $N_{\sigma'}$ of $\partial D_{\sigma}$ that satisfies
		\[
		v=0,\quad y\in  \partial D_{\sigma'},
		\]
		where 
		\[
		N_{\sigma'}:= \bigcup_{y\in \partial D_{\sigma'}}B(y,c_4)\cap \Omega,\quad c_4>0.
		\]
		Then, there exists a constant $C$ such that
		\begin{equation*}
			\sup_{\partial D_{\sigma'}} \bigg|\frac{\partial v}{\partial \nu}\bigg|
			\le C \sup_{N_{\sigma'}} |v|.
		\end{equation*}
		Furthermore, there exists a constant $C_b$ such that the following estimate holds:
		\begin{equation*}
			\int_{\partial D_{\sigma'}} \bigg|\frac{\partial v}{\partial \nu}\bigg| \, \du\sigma
			\le C_b \sup_{N_{\sigma'}}|v|.
		\end{equation*}
	\end{lemma}
	Applying the above lemma to $G(x;y)$ and $V_{\sigma}$ on $\partial D_{\sigma'}$, we derive the following estimate.
	\begin{theorem}\label{apthm:exp_decay}
		If the resonators $\mathcal{D}$ satisfy \Cref{asmp_geometry_assumption}, then there exist positive constants $C,c$ such that the coefficients of the capacitance operator $\mathcal{C}$ satisfy for $\sigma\neq \sigma'$:
		\begin{equation} \label{eqn:exp_decay}
			|\mathcal{C}(\sigma,\sigma')| \le C \frac{\eu^{-c\operatorname{dist}(D_{\sigma},D_{\sigma'})}}{\sqrt{\operatorname{dist}(D_{\sigma},D_{\sigma'})}}.
		\end{equation}
	\end{theorem}

	\section{Numerical methods}
	\label{apsec:numericalmethod}
	\subsection{Methods for calculating the capacitance operator}
	We adopt the boundary element method to calculate the capacitance operators. More precisely, for each patch $\mathsf{P}$, we solve the following boundary integral equation of the first kind:
	\begin{equation}\label{eqn:SingleLayerEqn}
		\mathcal{S}_{\mathsf{P}}[\phi]:=\int_{\partial \mathsf{P}}G(x-y)\phi(y)\, \du \sigma(y) = g(x),\quad x\in\partial \mathsf{P}.
	\end{equation}
	Here, the two-dimensional Green function $G$ is given by
	\[ G(x-y) = \frac{1}{2\pi}\ln|x-y|. \] 
	All connected components of $\partial \mathsf{P}$ are partitioned into straight segments and piecewise linear basis functions are used. Weakly singular integrals are evaluated using the logarithmic Gauss quadrature, while regular interactions are evaluated using the tensor-product Gauss-Legendre quadrature.
	The resulting linear system is of the form
	\[
	\mathbf{S}\boldsymbol{\phi} = \mathbf{M}\mathbf g,
	\]
	where $\mathbf{S}$ is the Galerkin matrix of the single-layer boundary integral operator, $\mathbf{M}$ is the mass matrix, and $\mathbf g$ is the Dirichlet datum.
	
	From the jump relation for the single layer potential, the exterior normal derivative is recovered as
	\begin{gather*}
		\begin{aligned}
			\left.\frac{\partial \mathcal{S}_{\mathsf{P}}[\phi]}{\partial \nu}\right|_{+}(x) &= \frac{1}{2}\phi(x)+\int_{\partial \mathsf{P}} \frac{\partial G}{\partial \nu_{x} }(x-y)\phi(y)\, \du \sigma(y)\\
			&:= \frac{1}{2}\phi(x) + (\mathcal{K}_{\mathsf{P}})^\ast[\phi],
		\end{aligned}
	\end{gather*}
	where $(\mathcal{K}_{\mathsf{P}})^\ast$ denotes the Neumann-Poincar\'e operator; see, for instance, \cite{Ammari2018}.
	When the Dirichlet datum $g$ is piecewise constant on each connected component of $\partial P$, the solution $\phi$ to \eqref{eqn:SingleLayerEqn} satisfies
	\[ \phi\in \ker(-\frac{1}{2}\operatorname{Id} + (\mathcal{K}_{\mathsf{P}})^\ast). \]
	Hence, $\phi$ can be used as a substitute for the exterior normal derivative. Numerical tests show that the error is of the order of $10^{-15}$.
	\subsection{Methods for calculating the eigenvalues} \label{subsec:pn}
	
	To compute the eigenvalues and the corresponding localized modes of the infinite-dimensional capacitance operator $\mathcal{C}$ and its patch approximation $\widetilde{\mathcal{C}}$, we employ the rectangular truncation framework recently established in \cite{Colbrook2019,Colbrook2023}, which ensures the absence of spectral pollution and also provides certifiable error bounds.
	
	In the context of capacitance operators, this method uses a rectangular truncation $ P_{\mathcal{N}'} \mathcal{C} P_{\mathcal{N}} \in \mathbb{R}^{\mathcal{N}' \times \mathcal{N}} $. Here, $P_{\mathcal{N}}$ is the projection on the computational domain $\mathcal{N}\subset \Sigma $. The set $\mathcal{N}'\supseteq \mathcal{N}$ is chosen so that the truncated matrix contains all significant interactions originating from the computational domain $\mathcal{N}$. For example, when $\Lambda = \mathbb{Z}^2$, we may choose 
	\begin{gather*}
		\begin{aligned}
			\mathcal{N}&= \{ -\widetilde{N}_{\mathrm{trunc}},\ldots ,\widetilde{N}_{\mathrm{trunc}} \}^2\times \{1,\ldots,K\},\\
			\mathcal{N}'&= \{ -(\widetilde{N}_{\mathrm{trunc}}+\widetilde{M}_{\mathrm{trunc}}),\ldots ,\widetilde{N}_{\mathrm{trunc}}+\widetilde{M}_{\mathrm{trunc}} \}^2\times\{1,\ldots,K\},
		\end{aligned}
	\end{gather*}
	for some $\widetilde{M}_{\mathrm{trunc}},\widetilde{N}_{\mathrm{trunc}}\in\mathbb{N}$. Note that when $\widetilde{M}_{\mathrm{trunc}}=0$, the rectangular truncation becomes the usual supercell truncation.
	\par 
	
	To find the eigenvalues, the algorithm relies on computing the smallest singular value $s_{\mathrm{min}}$ of the shifted rectangular matrix for $z\in\mathbb{C}$:
	\begin{equation} \label{def:fz}
		F_\mathcal{N}(z) := s_{\min}\big(P_{\mathcal{N}'} (\mathcal{C}-z\mathcal{M}) P_{\mathcal{N}}\big).
	\end{equation}
	The operator $\mathcal{M}$ is defined in \eqref{eqn:MassOperator}. 
	The value $F_\mathcal{N}(z)$ gives an upper bound for the distance between the parameter $z$ and the spectrum of the capacitance operator $\mathcal{C}$; see \cite{Colbrook2019}. By scanning $z$ over an interval of interest, the true eigenvalues correspond to the local minima of $F_\mathcal{N}(z)$.

	\bibliographystyle{amsplain}
	\bibliography{references}
	
\end{document}